\documentclass[12pt]{article}
\usepackage{pdfrender,xcolor}
\usepackage{graphicx,color}
\usepackage{amsmath,amsfonts,amssymb,amsthm}
\usepackage[hidelinks,colorlinks=true,linkcolor=blue,citecolor=blue]{hyperref}
\usepackage{rsfso}
\usepackage{comment}
\usepackage{enumerate}
\usepackage{pdfpages}

\usepackage[numbers]{natbib}
\newcommand{\ds}{\displaystyle }

\newcommand{\sr}[1]{{\mathcal #1}}
\newcommand{\dd}[1]{\mathbb{#1}}

\newcommand{\qvs}[1]{\big[ #1 \big]}

\newcommand{\ol}{\overline}

\newcommand{\sgn}{{\rm sgn}}

\newcommand{\eq}[1]{(\ref{eq:#1})}
\newcommand{\lem}[1]{Lemma~\ref{lem:#1}}
\newcommand{\cor}[1]{Corollary~\ref{cor:#1}}
\newcommand{\thr}[1]{Theorem~\ref{thr:#1}}

\newcommand{\dfn}[1]{Definition~\ref{dfn:#1}}

\newcommand{\cond}[1]{Condition~\ref{cond:#1}}
\newcommand{\rem}[1]{Remark~\ref{rem:#1}}
\newcommand{\fig}[1]{Figure~\ref{fig:#1}}
\newcommand{\app}[1]{Appendix~\ref{app:#1}}
\newcommand{\sectn}[1]{Section~\ref{sec:#1}}

\newcommand{\lemt}[1]{\ref{lem:#1}}

\newcommand{\thrt}[1]{\ref{thr:#1}}

\newcommand{\appt}[1]{\ref{app:#1}}

\newcommand{\sect}[1]{\ref{sec:#1}}

\newcommand{\pend}{\hfill \thicklines \framebox(6.6,6.6)[l]{}}

\newenvironment{proof*}[1]{\noindent {\sc  #1} \rm}{\pend}
\newenvironment{keyword}[1]{ {\bf #1} \rm}{}
\newenvironment{MSC}[1]{ {\bf #1} \rm}{}

\newtheorem{theorem}{Theorem}[section]
\newtheorem{lemma}{Lemma}[section]

\newtheorem{condition}{Condition}[section]
\newtheorem{remark}{Remark}[section]

\newtheorem{corollary}{Corollary}[section]
\newtheorem{definition}{Definition}[section]

\newcommand{\setnewcounter} {
\setcounter{subsection}{0}
\setcounter{equation}{0}
\setcounter{conjecture}{0}
\setcounter{assumption}{0}
\setcounter{question}{0}
\setcounter{definition}{0}
\setcounter{theorem}{0}
\setcounter{corollary}{0}
\setcounter{lemma}{0}
\setcounter{proposition}{0}
\setcounter{remark}{0}
}

\begin{document}
 \title{\Large \bf Multi-level reflecting Brownian motion on the half line and its stationary distribution}

%%%%%%%%%%%%%%
%   AUTHORS  %
%%%%%%%%%%%%%%
\author{Masakiyo Miyazawa\footnotemark}
\date{\today, updated}
%\date{}

\maketitle

\begin{abstract}
A semi-martingale reflecting Brownian motion is a popular process for diffusion approximations of queueing models including their networks. In this paper, we are concerned with the case that it lives on the nonnegative half-line, but the drift and variance of its Brownian component discontinuously change at its finitely many states. This reflecting diffusion process naturally arises from a state-dependent single server queue, studied by the author \cite{Miya2024}. Our main interest is in its stationary distribution, which is important for application. We define this reflecting diffusion process as the solution of a stochastic integral equation, and show that it uniquely exists in the weak sense. This result is also proved in a different way by \citet{AtarCastReim2022}. In this paper, we consider its Harris irreducibility and stability, that is, positive recurrence, and derive its stationary distribution under this stability condition. The stationary distribution has a simple analytic expression, likely extendable to a more general state-dependent SRBM. Our proofs rely on the generalized Ito formula for a convex function and local time.
\end{abstract}

\keyword{Keywords:}{reflecting Brownian motion, multi level, discontinuous diffusion coefficients, stationary distribution, stochastic integral equation, generalized Ito formula, Tanaka formula, Harris irreducibility.}

\MSC{MSC Classification:}{60H20, 60J25, 60J55, 60H30, 60K25}

\footnotetext[1]{Department of Information Sciences,
Tokyo University of Science, Noda, Chiba, Japan}

\section{Introduction}\label{sec:introduction}

We are concerned with a semi-martingale reflecting Brownian motion (SRBM for short) on the nonnegative half-line in which the drift and variance of its Brownian component discontinuously change at its finitely many states. The partitions of its state space which is separated by these states are called levels. This reflecting SRBM is denoted by $Z(\cdot) \equiv\{Z(t); t \ge 0\}$, and will be called a one-dimensional multi-level SRBM (see \dfn{1d-SRBM}). In particular, if the number of the levels is $k$, then it is called a one-dimensional $k$-level SRBM. Note that the one-dimensional 1-level SRBM is just the standard SRBM on the half line.

Let $Z(\cdot)$ be a one-dimesional $k$-level SRBM. This reflecting process for $k=2$ arises in the recent study of \citet{Miya2024} for asymptotic analysis of a state dependent single server queue, called $2$-level $GI/G/1$ queue, in heavy traffic. This queueing model was motivated by an energy saving problem on servers for internet.

In \cite{Miya2024}, it is conjectured that the reflecting process $Z(\cdot)$ for $k=2$ is the weak solution of a stochastic integral equation (see \eq{SIE-Z} in \sectn{problem}) and, if its stationary distribution exists, then this distribution agrees with the limit of the scaled stationary distribution of the $2$-level $GI/G/1$ queue under heavy traffic, which is obtained under some extra conditions in Theorem 3.1 of \cite{Miya2024}. While writing this paper, we have known that the weak existence of the solution is shown by \citet{AtarCastReim2022} for a more general model than we have studied here, and its uniqueness is proved in \cite[Lemma 4.1]{AtarCastReim2023} under one of the conditions of this paper.

We refer to these results of \citet{AtarCastReim2022,AtarCastReim2023} as \lem{SIE-Z0}. However, we here prove a slightly different lemma, \lem{SIE-Z}, which is restrictive for the existence but less restrictive for the uniqueness. \lem{SIE-Z} includes some further results which will be used. Furthermore, its proof is sample path based and different from that of \lem{SIE-Z0}. We then show in \lem{irreducible} that $Z(\cdot)$ is Harris irreducible, and give a necessary and sufficient condition for it to be positive recurrent in \lem{stability}. These three lemmas are bases for our stochastic analysis.

The main results of this paper are \thr{general k} and \cor{general k} for the $k$-level SRBM, which derive the stationary distribution of  $Z(\cdot)$ without any extra condition under the stability condition obtained in \lem{stability}. However, we first focus on the case for $k=2$ in \thr{k=2}, then consider the case for general $k$ in \thr{general k}. This is because the presentation and proof for general $k$ are notationally complicated while the proof for $k=2$ can be used with minor modifications for general $k$. The stationary distribution for $k=2$ is rather simple, it is composed of two mutually singular measures, one is truncated exponential or uniform on the interval $[0,\ell_{1}]$, and the other is exponential on $[\ell_{1},\infty)$, where $\ell_{1}$ is the right endpoint of the first level at which the variance of the Brownian component and drift of the $Z(\cdot)$ discontinuously change. One may easily guess these measures, but it is not easy to compute their weights by which the stationary distribution is uniquely determined (see \cite{Miya2024}). We resolve this computation problem using the local time of the semi-martingale $Z(\cdot)$ at $\ell_{1}$.

The key tools for the proofs of the lemmas and theorems are the generalized Ito formula for a convex function and local time due to \citet{Tana1963}. We also use the notion of a weak solution of a stochastic integral equation. These formulas and notions are standard in stochastic analysis nowadays (e.g, see \cite{ChunWill1990,CoheElli2015,Harr2013,Kall2001,KaraShre1998}), but they are briefly supplemented in the appendix because they play major roles in our proofs. 

This paper is made up by five sections. In \sectn{problem}, we formally introduce a one-dimensional reflecting SRBM with state-dependent Brownian component which includes the one-dimensional multi-level SRBM as a special case, and present preliminary results including Lemmas \lemt{SIE-Z}, \lemt{irreducible} and \lemt{stability}, which are proved in \sectn{preliminary}. Theorems \thrt{k=2} and \thrt{general k} are presented and proved in Section \sect{stationary}. Finally, a related problems and a generalization of \thr{general k} are discussed in \sectn{concluding}. In the appendix, the definitions of a weak solution for a stochastic integral equation and local time of a semi-martingale are briefly discussed in Sections \appt{weak-solution} and \appt{local-time}, respectively.

\section{Problem and preliminary lemmas}
\label{sec:problem}
\setnewcounter

Let $\sigma(x)$ and $b(x)$ be measurable positive and real valued functions, respectively, of $x \in \dd{R}$, where $\dd{R}$ is the set of all real numbers. We are interested in the solution $Z(\cdot) \equiv \{Z(t); t \ge 0\}$ of the following stochastic integral equation, SIE for short. 
\begin{align}
\label{eq:SIE-Z}
  Z(t) = Z(0) & + \int_{0}^{t} \sigma(Z(u)) dW(u) + \int_{0}^{t} b(Z(u)) du + Y(t) \ge 0, \qquad t \ge 0,
\end{align}
where $W(\cdot)$ is the standard Brownian motion, and $Y(\cdot) \equiv \{Y(t); t \ge 0\}$ is a non-deceasing process satisfying that $\int_{0}^{t} 1(Z(u) > 0) dY(u) = 0$ for $t \ge 0$. We refer to this $Y(\cdot)$ as a regulator. The state space of $Z(\cdot)$ is denoted by $S \equiv \dd{R}_{+}$, where $\dd{R}_{+} = \{x \in \dd{R}; x \ge 0\}$.

As usual, we assume that all continuous-time processes are defined on stochastic basis $(\Omega,\sr{F}, \dd{F}, \dd{P})$, and right-continuous with left-limits and $\dd{F}$-adapted, where $\dd{F} \equiv \{\sr{F}_{t}; t \ge 0\}$ is a right-continuous filtration. Note that there are two kinds of solutions, strong and weak ones, for the SIE \eq{SIE-Z}. See \app{weak-solution} for their definitions. In this paper, we call weak solution simply by solution unless stated otherwise.

If functions $\sigma(x)$ and $b(x)$ are Lipschitz continuous and their squares are bounded by $K(1+x^{2})$ for some constant $K > 0$, then the SIE \eq{SIE-Z} has a unique solution even for the multidimensional SRBM which lives on a convex region (see \cite[Theorem 4.1]{Tana1979}). However, we are interested in the case that $\sigma(x)$ and $b(x)$ discontinuously change. In this case, any solution $Z(\cdot)$ may not exist in general, so we need some condition. As we discussed in \sectn{introduction}, we are particularly interested when they satisfy the following conditions. Let $\ol{\dd{R}} = \dd{R} \cup \{-\infty,+\infty\}$.
\begin{condition}
\label{cond:multi-level}
There are an integer $k \ge 2$ and a strictly increasing sequence $\{\ell_{j} \in \ol{\dd{R}}; j = 0,1,\ldots,k\}$ satisfying $\ell_{0} = - \infty$, $\ell_{j} > 0$ for $j=1,2,\ldots,k-1$ and $\ell_{k} = \infty$ such that functions $\sigma(x) > 0$ and $b(x) \in \dd{R}$ for $x \in \dd{R}$ are given by
\begin{align}
\label{eq:multi-level}
 & \sigma(x) = \sum_{j=1}^{k} \sigma_{j} 1(\ell_{j-1} \le x < \ell_{j}), \qquad b(x) = \sum_{j=1}^{k} b_{j} 1(\ell_{j-1} \le x < \ell_{j}),
\end{align}
where $1(\cdot)$ is the indicator function of proposition ``$\cdot$'', and $\sigma_{j} > 0$, $b_{j} \in \dd{R}$ for $j = 1,2,\ldots,k$ are constants.
\end{condition}

Since $Z(t)$ of \eq{SIE-Z} is nonnegative, $\sigma(x)$ and $b(x)$ are only used for $x \ge 0$ in \eq{SIE-Z}. Taking this into account, we partition the state space $S \equiv \dd{R}_{+}$ of $Z(\cdot)$ by $\ell_{1}, \ell_{2}, \ldots, \ell_{k-1}$ under \cond{multi-level} as follows.
\begin{align}
\label{eq:partition}
  S_{1} = [0,\ell_{1}), \quad S_{j} = [\ell_{j-1}, \ell_{j}), \;\; j=2,3,\ldots, k-1, \quad S_{k} = [\ell_{k-1},+\infty).
\end{align}
We call these $S_{j}$'s levels. Note that $\sigma(x)$ and $b(x)$ are constants in $x$ on each level, and they may discontinuously change at state $\ell_{j}$ for $j=1,2,\ldots, k-1$ under \cond{multi-level}.

We start with the existence of the solution $Z(\cdot)$ of \eq{SIE-Z}, which will be implied by the existence of the solution $X(\cdot) \equiv \{X(t); t \ge 0\}$ of the following stochastic integral equation in the weak sense.
\begin{align}
\label{eq:SIE-X}
  X(t) = X(0) & + \int_{0}^{t} \sigma(X(u)) dW(u) + \int_{0}^{t} b(X(u)) du, \qquad t \ge 0.
\end{align}
See \rem{existence}. Taking this into account, we will also consider the condition below for the existence of the weak solution $(X(\cdot), W(\cdot))$.
\begin{condition}
\label{cond:VD2}
The functions $\sigma(x)$ and $b(x)$ are measurable functions satisfying that 
\begin{align}
\label{eq:VD-bound}
 & \inf_{x \in \dd{R}} \sigma(x) > 0, \qquad \sup_{x \in \dd{R}} |b(x)| < \infty.
\end{align}
\end{condition}

For the unique existence of the weak solution $(X(\cdot), W(\cdot))$, \cond{VD2} is further weakened to \cond{minimal VD} by Theorem 5.15 of \cite{KaraShre1998} (see \sectn{weaker conditions} for its details). However, the latter condition is quite complicated. So, we take the simpler condition \eq{VD-bound}, which is sufficient for our arguments.

\begin{definition}
\label{dfn:1d-SRBM}
The solutions $Z(\cdot)$ of \eq{SIE-Z} under \cond{multi-level} is called a one-dimensional multi-level SRBM, in particular, called one-dimensional $k$-level SRBM if it has $k$ levels, namely, the total number of partitions of \eq{partition} is $k$, while it under \cond{VD2} is called a one-dimensional state-dependent SRBM with bounded drifts.
\end{definition}

Using the weak solution $(X(\cdot), W(\cdot))$ of \eq{SIE-X}, \citet{AtarCastReim2022,AtarCastReim2023} proves:
\begin{lemma}[Lemma 4.3 of \cite{AtarCastReim2022} and Lemma 4.1 of \cite{AtarCastReim2023}]
\label{lem:SIE-Z0}
(i) Under \cond{VD2}, the stochastic integral equation \eq{SIE-Z} has a weak solution such that $Y(t)$ is continuous in $t \ge 0$. (ii) Under \cond{multi-level}, the solution is weakly unique.
\end{lemma}

The proof of (i) is easy  (see \rem{existence}) while the proof of (ii) is quite technical. Instead of this lemma, we will use the following lemma, in which (i) and (ii) of \lem{SIE-Z0} are proved under more restrictive and less restrictive conditions, respectively.

\begin{lemma}
\label{lem:SIE-Z}
Under \cond{VD2}, if there are constants $\sigma_{1}, \ell_{1} > 0$ and $b_{1} \in \dd{R}$ such that
\begin{align}
  \label{eq:VD2+}
  & \sigma(x) = \sigma_{1} > 0, b(x) = b_{1}, \qquad \forall x < \ell_{1},
\end{align}
then the stochastic integral equation \eq{SIE-Z} has a unique weak solution such that $Y(t)$ is continuous in $t \ge 0$ and $Z(\cdot)$ is a strong Markov process.
\end{lemma}
We prove this lemma in \sectn{existence}, which is different from the proof of \lem{SIE-Z0} by \citet{AtarCastReim2022,AtarCastReim2023}.

The main interest of this paper is to derive the stationary distribution of the $Z(\cdot)$ for the one-dimensional multi-level SRBM under an appropriate stability condition. Since this  reflecting diffusion process satisfies the conditions of \lem{SIE-Z}, $Z(\cdot)$ is a strong Markov process. Hence, our first task for deriving its stationary distribution is to consider its irreducibility and positive recurrence. To this end, we introduce Harris irreducibility and recurrence following \cite{MeynTwee1993a}. Let $\sr{B}(\dd{R}_{+})$ be the Borel field, that is, the minimal $\sigma$-algebra on $\dd{R}_{+}$ which contains all open sets of $\dd{R}_{+}$. Then, a real valued process $X(\cdot)$ which is right-continuous with left-limits is called Harris irreducible if there is a non-trivial $\sigma$-finite measure $\psi$ on $(\dd{R}_{+}, \sr{B}(\dd{R}_{+}))$ such that, for $B \in \sr{B}(\dd{R}_{+})$, $\psi(B) > 0$ implies
\begin{align}
\label{eq:Harris1}
  \dd{E}_{x}\left[\int_{0}^{\infty} 1(X(u) \in B) du\right] > 0, \qquad \forall x \in \dd{R}_{+},
\end{align}
while it is called Harris recurrent if \eq{Harris1} can be replaced by
\begin{align}
\label{eq:Harris2}
  \dd{P}_{x}\left[ \int_{0}^{\infty} 1(X(u) \in B) du = \infty \right] = 1, \qquad \forall x \in \dd{R}_{+},
\end{align}
where $\dd{P}_{x}(A) = \dd{P}(A|X(0)=x)$ for $A \in \sr{F}$, and $\dd{E}_{x}[H|X(0)=x]$ for a random variable $H$.

Harris conditions \eq{Harris1} and \eq{Harris2} are related to hitting times. Define the hitting time at a subset of the state space $S$ as
\begin{align}
\label{eq:tau-B}
  \tau_{B} = \inf \{t \ge 0; X(t) \in B\}, \qquad B \in \sr{B}(\dd{R}_{+}),
\end{align}
where $\tau_{B} = \infty$ if $X(t) \not\in B$ for all $t \ge 0$. We denote $\tau_{B}$ simply by $\tau_{a}$ for $B = \{a\}$. Then, it is known that Harris recurrent condition \eq{Harris2} is equivalent to
\begin{align}
\label{eq:Harris3}
  \dd{P}_{x}[\tau_{B} < \infty] = 1, \qquad \forall x \ge 0,
\end{align}
See Theorem 1 of \cite{KaspMand1994} (see also \cite{MeynTwee1993a}) for the proof of this equivalence. However, Harris irreducible condition \eq{Harris1} may not be equivalent to $\dd{P}_{x}[\tau_{B} < \infty]>0$. In what follows, $\tau_{B}$ is defined for the process to be discussed unless stated otherwise.

Using those notions and notations, we present the following basic facts for the one-dimensional state-dependent SRBM $Z(\cdot)$ with bounded drifts, where $Z(\cdot)$ is strong Markov by \lem{SIE-Z}.

\begin{lemma}
\label{lem:irreducible}
For the one-dimensional state-dependent SRBM with bounded drifts, if the condition \eq{VD2+} of \lem{SIE-Z} is satisfied, (i) it is Harris irreducible, and (ii) $\dd{E}_{x}[\tau_{a}] < \infty$ for $0 \le x < a$.
\end{lemma}

\begin{remark}
\label{rem:irreducible}
(ii) is not surprising because we can intuitively see that the drift is pushed to the upward direction by reflection at the origin and the positive-valued variances.
\end{remark}

\begin{lemma}
\label{lem:stability}
For the one-dimensional state-dependent SRBM with bounded drifts, if the condition \eq{VD2+} of \lem{SIE-Z} is satisfied and if there are constants $\ell_{*}, b_{*}$ and $\sigma_{*}$ such that
\begin{align}
\label{eq:VD3}
    & \sigma(x) = \sigma_{*} > 0, b(x) = b_{*}, \qquad \forall x \ge \ell_{*} > \ell_{1}, 
\end{align}
then $Z(\cdot)$ has a stationary distribution if and only if $b_{*} < 0$. In particular, the one-dimensional $k$-level SRBM has a stationary distribution if and only if $b_{k} < 0$. 
\end{lemma}

These lemmas may be intuitively clear, but their proofs may have own interests because they are not immediate and we observe that the Ito formula nicely work. So we prove Lemmas \lemt{irreducible} and \lemt{stability} in Sections \sect{irreducible} and \sect{stability}, respectively. We are now ready to study the stationary distribution of $Z(\cdot)$.

\section{Stationary distribution of multi-level SRBM}
\label{sec:stationary}
\setnewcounter

We are concerned with the multi-level SRBM. Denote the number of its levels by $k$. We first introduce basic notations. Let $\sr{N}_{k} = \{1,2,\ldots,k\}$, and define
\begin{align*}
  \beta_{j} = 2 b_{j}/\sigma_{j}^{2}, \quad j \in \sr{N}_{k}.
\end{align*}
In this section, we derive the stationary distribution of the one-dimensional $k$-level SRBM for arbitrary $k \ge 2$. We first focus on the case for $k=2$ because this is the simplest case but its proof contains all ideas will be used for general $k$.

\subsection{Stationary distribution for $k=2$}
\label{sec:k=2}

Throughout \sectn{k=2}, we assumed that $k=2$.

\begin{theorem}[The case for $k=2$]
\label{thr:k=2}
The $Z(\cdot)$ of the one-dimensional $2$-level SRBM has a stationary distribution if and only if $b_{2} < 0$, equivalently, $\beta_{2} < 0$. Assume that $b_{2} < 0$, and let $\nu$ be the stationary distribution of $Z(t)$, then $\nu$ is unique and has a probability density function $h$ which is given below.\\
(i) If $b_{1} \not= 0$, then
\begin{align}
\label{eq:Z1-h}
  h(x) = d_{11} h_{11}(x) + d_{12} h_{2}(x), \qquad x \ge 0,
\end{align}
where $h_{11}$ and $h_{2}$ are probability density functions defined as
\begin{align}
\label{eq:h(b1not0)}
 & h_{11}(x) = \frac {e^{\beta_{1} (x - \ell_{1})} \beta_{1}}{e^{\beta_{1} \ell_{1}} - 1}1(0 \le x < \ell_{1}),
\qquad
  h_{2}(x) = - \beta_{2} e^{\beta_{2} (x - \ell_{1})} 1(x \ge \ell_{1}),
\end{align}
and $d_{1j}$ for $j=1,2$ are positive constants defined by
\begin{align}
\label{eq:d(b1not0)}
   d_{11} = \frac {b_{2} (e^{-\beta_{1} \ell_{1}} - 1)} {b_{1} + b_{2} (e^{-\beta_{1} \ell_{1}} - 1)}, \qquad d_{12} = 1 -d_{11} = \frac {b_{1}} {b_{1} + b_{2} (e^{-\beta_{1} \ell_{1}} - 1)}.
\end{align}
(ii) If $b_{1} = 0$, then
\begin{align}
\label{eq:Z0-h}
  h(x) = d_{01} h_{01}(x) + d_{02} h_{2}(x),
\end{align}
where $h_{2}$ is defined in \eq{h(b1not0)}, and
\begin{align}
\label{eq:h(b1=0)}
 & h_{01}(x) = \frac 1{\ell_{1}} 1(0 \le x < \ell_{1}),\\
\label{eq:d(b1=0)}
 &  d_{01} = \frac {-2b_{2} \ell_{1}} {\sigma_{1}^{2} - 2b_{2} \ell_{1}}, \qquad d_{02} = 1 -d_{01} = \frac {\sigma_{1}^{2}} {\sigma_{1}^{2} - 2b_{2} \ell_{1}}.
\end{align}
\end{theorem}

\begin{remark}
\label{rem:k=2}
(a) \eq{h(b1=0)} and \eq{d(b1=0)} are obtained from \eq{h(b1not0)} and \eq{d(b1not0)} by letting $b_{1} \to 0$.\\
(b) Assume that $Z(\cdot)$ is a stationary process, and define the moment generating functions (mgf for short):
\begin{align*}
 & \varphi(\theta) = \dd{E}[e^{\theta Z(1)}], \\
 & \varphi_{1}(\theta) = \dd{E}[e^{\theta Z(1)}1(0 \le Z(1) < \ell_{1})], \qquad \varphi_{2}(\theta) = \dd{E}[e^{\theta Z(1)}1(Z(1) \ge \ell_{1})].
\end{align*}
Here, $\varphi(\theta)$ and $\varphi_{2}(\theta)$ are finite for $\theta \le 0$, and $\varphi_{1}(\theta)$ does so for $\theta \in \dd{R}$. However, all of them are uniquely identified for $\theta \le 0$ as Laplace transforms. So, in what follows, we always assume that $\theta \le 0$ unless stated otherwise.

For $i=0,1$, let $\widehat{h}_{i1}$ and $\widehat{h}_{2}$ be the moment generating functions of $h_{i1}$ and $h_{2}$, respectively, then
\begin{align}
\label{eq:Lh1}
 & \widehat{h}_{i1}(\theta) = 
 \begin{cases}
\frac {e^{\theta \ell_{1}} - e^{-\beta_{1} \ell_{1}}} {\beta_{1} + \theta} \frac {\beta_{1}} {1 - e^{-\beta_{1} \ell_{1}}} 1(\theta \not= -\beta_{1}), \quad  & i = 1,\\
 \frac {e^{\theta \ell_{1}} - 1} {\ell_{1} \theta} 1(\theta \not= 0), & i = 0,
\end{cases}\\
\label{eq:Lh2}
 & \widehat{h}_{2}(\theta) = e^{\theta \ell_{1}} \frac {\beta_{2}}{\beta_{2} + \theta}, \qquad \theta \le 0,
\end{align}
where the singular points $\theta = - \beta_{1}, 0$ in \eq{Lh1} are negligible to determine $h_{i1}$, so we take the convention that $\widehat{h}_{i1}(\theta)$ exists for these $\theta$.

Hence, \eq{Z1-h} for $b_{1} \not= 0$ and \eq{Z0-h} for $b_{1} = 0$ are equivalent to
\begin{align}
\label{eq:phi-normalize}
 & \varphi_{1}(\theta)/\varphi_{1}(0) = \begin{cases}
\widehat{h}_{11}(\theta), & b_{1} \not= 0,\\
\widehat{h}_{01}(\theta), & b_{1} = 0,
\end{cases}, \qquad \varphi_{2}(\theta)/\varphi_{2}(0) = \widehat{h}_{2}(\theta),\\
\label{eq:d11}
 & \varphi_{1}(0) = d_{11}, \quad \varphi_{2}(0) = d_{12}, \; \mbox{ for } \; b_{1} \not= 0,\\
\label{eq:d01}
 & \varphi_{1}(0) = d_{01}, \quad \varphi_{2}(0) = d_{02}, \: \mbox{ for } \; b_{1} = 0.
\end{align}
Thus, \thr{k=2} is proved by showing these equalities. 
\end{remark}

\begin{remark}
\label{rem:compatibility}
\citet{Miya2024} conjectures that the diffusion scaled process limit of the queue length of the 2-level $GI/G/1$ queue in heavy traffic is the solution of the stochastic integral equation of (5.2) in \cite{Miya2024}. This stochastic equation corresponds to \eq{SIE-Z}, but $\ell_{1}$, $b_{i}$ and $\sigma_{i}$ of the present paper needs to replace by $\ell_{0}$, $-b_{i}$, $\sqrt{c_{i}} \sigma_{i}$ for $i=1,2$, respectively. Under these replacements, $\beta_{i}$ also needs to replace by $- 2b_{i}/(c_{i}\sigma_{i}^{2})$. Then, it follows from \eq{d(b1not0)}, \eq{d(b1=0)}, \eq{d11} and \eq{d01} that, under the setting of \citet{Miya2024}, for $b_{1} \not= 0$,
 \begin{align*}
%\label{eq:phi(0)-11}
  \varphi_{1}(0) = d_{11} = \frac {c_{1} b_{2} (e^{\beta_{1} \ell_{1}} - 1) } {c_{2} b_{1} + c_{1} b_{2} (e^{\beta_{1} \ell_{1}} - 1)}, \quad \varphi_{2}(0) = d_{12} = \frac {c_{2} b_{1}} {c_{2} b_{1} + c_{1} b_{2} (e^{\beta_{1} \ell_{1}} - 1)},
\end{align*}
and, for $b_{1} = 0$,
\begin{align*}
%\label{eq:phi(0)-01}
 & \varphi_{1}(0) = d_{01} = \frac {2b_{2} \ell_{1}} {c_{2}\sigma_{1}^{2} + 2b_{2} \ell_{1}}, \quad \varphi_{2}(0) = d_{02} = \frac {c_{2}\sigma_{1}^{2}} {c_{2}\sigma_{1}^{2} + 2b_{2} \ell_{1}}.
\end{align*}
Hence, the limiting distributions in (ii) of Theorem 3.1 of \cite{Miya2024} are identical with the stationary distributions in \thr{k=2} here. Note that the limiting distributions in \cite{Miya2024} are obtained under some extra conditions, which are not needed for \thr{k=2}.
\end{remark}

\subsection{Proof of \thr{k=2}}
\label{sec:proof-T31}

By \rem{k=2}, it is sufficient to show \eq{phi-normalize}, \eq{d11} and \eq{d01} for the proof of \thr{k=2}. We will do it  in three steps.

\subsubsection{1st step of the proof}
\label{sec:1st}

 In this subsection, we derive two stochastic equations from \eq{SIE-Z}. For this, we use the generalized Ito formulas for a continuous semi-martingale $X(\cdot)$ with finite quadratic variations $[X]_{t}$ for all $t \ge 0$. For a convex test function $f$, this Ito formula is given by
\begin{align}
\label{eq:G-Ito}
  f(X(t)) = f(X(0)) + \int_{0}^{t} f'(X(u)-) dX(u) + \frac 12 \int_{0}^{\infty} L_{x}(t) \mu_{f}(dx), \qquad t \ge 0,
\end{align}
where $L_{x}(t)$ is the local time of $X(\cdot)$ which is right-continuous in $x \in \dd{R}$, and $\mu_{f}$ on $(\dd{R}_{+}, \sr{B}(\dd{R}_{+}))$ is a measure on $(\dd{R},\sr{B}(\dd{R})$, defined by
\begin{align}
\label{eq:Mf}
  \mu_{f}([x,y)) = f'(y-) - f'(x-), \qquad x < y \mbox{ in } \dd{R},
\end{align}
where  $f'(x-)$ is the left derivative of $f$ at $x$. See \app{local-time} for the definition of local time and more about its connection to the generalized Ito formula \eq{G-Ito}.

Furthermore, if $f(x)$ is twice differentiable, then \eq{G-Ito} can be written as
\begin{align}
\label{eq:Ito}
  f(X(t)) = f(X(0)) + \int_{0}^{t} f'(X(u)) dX(u) + \frac 12 \int_{0}^{\infty} f''(X(u)) d[X]_{u}, \qquad t \ge 0,
\end{align}
which is well known Ito formula.

In our application of the generalized Ito formula, we first take the following convex function $f$ with parameter $\theta \le 0$ as a test function.
\begin{align}
\label{eq:SIE-test1}
  f(x) = e^{\theta x} 1(x < \ell_{1}) + e^{\theta \ell_{1}}1(x \ge \ell_{1}). \qquad x \in \dd{R}.
\end{align}
Since $f'(\ell_{1}+) = 0$ and $f'(\ell_{1}-) = \theta e^{\theta \ell_{1}}$, it follows from \eq{Mf} that
\begin{align*}
  \mu_{f}(\{\ell_{1}\}) = \lim_{\varepsilon \downarrow 0} f'((\ell_{1}+\varepsilon) -) - f'(\ell_{1}-) = f'(\ell_{1}+) - f'(\ell_{1}-) = - \theta e^{\theta \ell_{1}}.
\end{align*}
On the other hand, $f''(x) = \theta^{2} e^{\theta x}$ for $x < \ell_{1}$. Hence,
\begin{align*}
 & \int_{0}^{\infty} L_{x}(t) \mu_{f}(dx) = \int_{0}^{\infty} L_{x}(t) f''(x-) 1(x < \ell_{1}) dx + L_{\ell_{1}}(t) \mu_{f}(\{\ell_{1}\}).
\end{align*}
Then, applying local time characterization \eq{L-time1} to this formula, we have
\begin{align}
\label{eq:L-time-mu-f}
 & \int_{0}^{\infty} L_{x}(t) \mu_{f}(dx) = \theta^{2} \int_{0}^{t} e^{\theta Z(u)} 1(0 \le Z(u) < \ell_{1}) d[Z]_{u} - \theta e^{\theta \ell_{1}} L_{\ell_{1}}(t).
\end{align}

We next compute the quadratic variation $\qvs{Z}_{t}$ of $Z(\cdot)$. Define $M(\cdot) \equiv \{M(t); t \ge 0\}$ by
\begin{align}
\label{eq:Z-M}
  M(t) \equiv \int_{0}^{t} \sigma(Z(u)) dW(u), \qquad t \ge 0,
\end{align}
then $M(\cdot)$ is a martingale. Denote the quadratic variations of $Z(\cdot)$ and $M(\cdot)$, respectively, by $\qvs{Z}_{t}$ and $\qvs{M}_{t}$. Since $Z(t)$ and $Y(t)$ are continuous in $t$, it follows from \eq{SIE-Z} that
\begin{align}
\label{eq:M-qv}
  \qvs{Z}_{t} & = \qvs{M}_{t} = \int_{0}^{t} \sigma^{2}(Z(u)) du \nonumber\\
  & = \int_{0}^{t} \sigma_{1}^{2} 1(0 \le Z(u) < \ell_{1}) + \sigma_{2}^{2} 1(Z(u) \ge \ell_{1})] du,  \qquad t \ge 0.
\end{align}

Hence, from $f'(\theta) = \theta e^{\theta x} 1(x < \ell_{1})$, \eq{L-time-mu-f} and \eq{Z-M}, the generalized Ito formula \eq{G-Ito} becomes
\begin{align}
\label{eq:G-Ito1}
  f(Z(t)) & = f(Z(0)) + \int_{0}^{t} \theta e^{\theta Z(u)} 1(0 \le Z(u) < \ell_{1}) \sigma_{1} dW(u)  \nonumber\\
 & \quad + \int_{0}^{t} \Big(b_{1} \theta + \frac 12 \sigma^{2}_{1} \theta^{2}\Big) e^{\theta Z(u)} 1(0 \le Z(u) < \ell_{1}) du + \theta Y(t) \nonumber\\
 & \quad  - \frac 12 \theta e^{\theta \ell_{1}} L_{\ell_{1}}(t), \qquad t \ge 0, \theta \le 0.
\end{align}

We next applying Ito formula for the test function $f(x) = e^{\theta x}$ to \eq{SIE-Z}. In this case, we use Ito formula \eq{Ito} because $f(x)$ is twice continuously differentiable. Then, we have, for $\theta \le 0$,
\begin{align}
\label{eq:G-Ito2}
 f(Z(t)) & = f(Z(0)) + \int_{0}^{t} \theta e^{\theta Z(u)} \sigma(u) dW(u)  \nonumber\\
 & \quad + \int_{0}^{t} \Big(b(Z(u)) \theta + \frac 12 \sigma^{2}(Z(u)) \theta^{2}\Big) e^{\theta Z(u)} du + \theta Y(t).
\end{align}

\subsubsection{2nd step of the proof}
\label{sec:2nd}

The first statement of \thr{k=2} is immediate from \lem{stability}. Hence, under $b_{2} > 0$, we can assume that $Z(\cdot)$ is a stationary process by taking its stationary distribution for the distribution of $Z(0)$. In what follows, this is always assumed.

Recall the moment generating functions $\varphi, \varphi_{1}$ and $\varphi_{2}$, which are defined in \rem{k=2}. We first consider the stochastic integral equation \eq{G-Ito1} to compute $\varphi_{1}$. Since $\dd{E}[L_{\ell_{1}}(1)]$ is finite by \lem{LT-time}, taking the expectation of \eq{G-Ito1} for $t = 1$ and $\theta \le 0$ yields
\begin{align}
\label{eq:RBM-BAR1}
 & \frac 12 \sigma_{1}^{2} ( \beta_{1} + \theta) \varphi_{1}(\theta) - \frac 12 e^{\theta \ell_{1}} \dd{E}[L_{\ell_{1}}(1)] + \dd{E}[Y(1)] = 0,
\end{align}
because $\beta_{1} \sigma_{1}^{2} = 2 b_{1}$. Note that this equation implies that $\dd{E}[Y(1)]$ is also finite.

Using \eq{RBM-BAR1}, we consider $\varphi_{1}(\theta)$ separately for $b_{1} \not= 0$ and $b_{1} = 0$. First, assume that $b_{1} \not= 0$. Then, from \eq{RBM-BAR1} and $\beta_{1} > 0$, we have
\begin{align}
\label{eq:phi11}
 & \varphi_{1}(\theta) =  \frac {e^{\theta \ell_{1}} \dd{E}[L_{\ell_{1}}(1)] - 2 \dd{E}[Y(1)]}{\sigma_{1}^{2} \left(\beta_{1} + \theta\right)}, \qquad \theta \not= \beta_{1}.
\end{align}
This equation can be written as
\begin{align}
\label{eq:phi12}
 & \varphi_{1}(\theta) =  \frac {e^{-\beta_{1} \ell_{1}} \dd{E}[L_{\ell_{1}}(1)] - 2 \dd{E}[Y(1)]}{\sigma_{1}^{2} \left(\beta_{1} + \theta\right)} + \frac {(e^{\theta \ell_{1}} - e^{-\beta_{1} \ell_{1}})\dd{E}[L_{\ell_{1}}(1)]}{\sigma_{1}^{2}(\beta_{1} + \theta)} \frac {} {}.
\end{align}
Observe that the first term in the right-hand side of \eq{phi12} is proportional to the moment generating function (mgf) of the signed measure on $[0,\infty)$ whose density function is exponential while its second term is the mgf of a measure on $[0,\ell_{1}]$, but the left-hand side of \eq{phi12} is the mgf of a probability measure on $[0,\ell_{1})$. Hence, we must have
\begin{align}
\label{eq:L-time3}
  2 \dd{E}[Y(1)] = e^{-\beta_{1} \ell_{1}} \dd{E}[L_{\ell_{1}}(1)],
\end{align}
and therefore \eq{phi12} yields
\begin{align}
\label{eq:phi13}
 & \varphi_{1}(\theta) =  \widehat{h}_{11}(\theta) \frac {(1 - e^{-\beta_{1} \ell_{1}}) \dd{E}[L_{\ell_{1}}(1)]} {\beta_{1} \sigma_{1}^{2}}, \qquad \theta \le 0,
\end{align}
where $\widehat{h}_{11}(\theta)$ is defined in \eq{Lh1}, but also exists for $\theta = - \beta_{1}$ by our convention.

We next assume that $b_{1} = 0$. In this case, $\beta_{1} = 0$, and it follows from \eq{phi11} that
\begin{align}
\label{eq:phi14}
  \varphi_{1}(\theta) = \frac {e^{\theta \ell_{1}} - 1} {\ell_{1} \theta} \times \frac {\ell_{1} \dd{E}[L_{\ell_{1}}(1)]} {\sigma_{1}^{2}} + \frac {\dd{E}[L_{\ell_{1}}(1)] - 2\dd{E}[Y(1)]} {\sigma_{1}^{2} \theta}, \qquad \theta \le 0.
\end{align}
Since $\frac {e^{\theta \ell_{1}} - 1} {\ell_{1} \theta}$ is the mgf of the uniform distribution on $[0,\ell_{1})$, by the same reason as in the case for $b_{1} \not= 0$, we must have
\begin{align*}
  2\dd{E}[Y(1)] = \dd{E}[L_{\ell_{1}}(1)].
\end{align*}
Note that this equation is identical with \eq{L-time3} for $b_{1} = 0$. Furthermore, $\lim_{\theta \uparrow 0} \frac {e^{\theta \ell_{1}} - 1} {\ell_{1} \theta} = 1$ and $\lim_{\theta \uparrow 0} \varphi_{1}(\theta) = \varphi_{1}(0)$. Hence, \eq{phi14} implies that
\begin{align}
\label{eq:phi5}
  \varphi_{1}(\theta) = \widehat{h}_{01}(\theta) \frac {\ell_{1} \dd{E}[L_{\ell_{1}}(1)]} {\sigma_{1}^{2}}, \qquad \theta \le 0,
\end{align}
by our convention for $\widehat{h}_{01}(0)$ similar to $\widehat{h}_{11}(-\beta_{1})$. Thus, we have the following lemma.
\begin{lemma}
\label{lem:phi1}
The mgf $\varphi_{1}$ is obtained as
\begin{align}
\label{eq:phi1}
  \varphi_{1}(\theta) = \begin{cases}
 \ds \widehat{h}_{11}(\theta) \frac {(1 - e^{\beta_{1} \ell_{1}}) \dd{E}[L_{\ell_{1}}(1)]} {\beta_{1} \sigma_{1}^{2}}, \quad & b_{1} \not= 0,  \medskip\\
 \ds \widehat{h}_{01}(\theta) \frac {\ell_{1} \dd{E}[L_{\ell_{1}}(1)]} {\sigma_{1}^{2}}, \quad &  b_{1} = 0,
\end{cases}
\qquad \theta \le 0.
\end{align}
\end{lemma}

We next consider the stochastic integral equation \eq{G-Ito2} to derive $\varphi_{2}(\theta)$. In this case, we use \eq{G-Ito2}. Note that $\varphi_{1}(\theta)$ and $\varphi_{2}(\theta)$ are finite for $\theta \le 0$. Hence, taking the expectations of both sides of \eq{G-Ito2} for $t=1$ and $\theta \le 0$ yields
\begin{align}
\label{eq:RBM-BAR2}
  \frac 12 \sum_{i=1,2} \sigma_{i}^{2} \left(\beta_{i} + \theta \right) \varphi_{i}(\theta) + \dd{E}[Y(1)] = 0, \qquad \theta \le 0.
\end{align}
Substituting $\frac 12 \sigma_{1}^{2} \left(\beta_{1} + \theta \right)$ of \eq{RBM-BAR1} and $\dd{E}[Y(1)]$ of \eq{L-time3} into this equation, we have 
\begin{align}
\label{eq:RBM-BAR3}
 & \sigma_{2}^{2} \left(\beta_{2} + \theta \right) \varphi_{2}(\theta) + e^{\theta \ell_{1}} \dd{E}[L_{\ell_{1}}(1)] = 0, \qquad \theta \le 0.
\end{align}
The following lemma is immediate from this equation since $\beta_{2} < 0$.

\begin{lemma}
\label{lem:phi2}
The mgf $\varphi_{2}$ is obtained as
\begin{align}
\label{eq:phi21}
 & \varphi_{2}(\theta) = \widehat{h}_{2}(\theta) \frac {\dd{E}[L_{\ell_{1}}(1)]}{-\beta_{2} \sigma_{2}^{2}}, \qquad \theta \le 0,
\end{align}
where recall that $\widehat{h}_{2}$ is defined by \eq{Lh2}.
\end{lemma}

\subsubsection{3rd step of the proof}
\label{sec:3rd}

We now prove \eq{phi-normalize}, \eq{d11} and \eq{d01}. Since $\widehat{h}_{11}(0) = \widehat{h}_{01}(0) = \widehat{h}_{2}(0) = 1$, \eq{phi-normalize} is immediate from Lemmas \lemt{phi1} and \lemt{phi2}. To prove \eq{d11}, assumed that $b_{1} \not = 0$. In this case, from \eq{phi13} and \eq{phi21}, we have
\begin{align*}
 & \varphi_{1}(0) = \frac {(1 - e^{-\beta_{1} \ell_{1}}) } {\beta_{1} \sigma_{1}^{2}} \dd{E}[L_{\ell_{1}}(1)],\quad
 \varphi_{2}(0) = \frac {\dd{E}[L_{\ell_{1}}(1)]}{-\beta_{2} \sigma_{2}^{2}}.
\end{align*}
Taking the ratios of both sides, we have
\begin{align*}
  \frac {\varphi_{1}(0)} {\varphi_{2}(0)} = \frac {-\beta_{2} \sigma_{2}^{2} (1 - e^{-\beta_{1} \ell_{1}}) } {\beta_{1} \sigma_{1}^{2}}.
\end{align*}
Since $\varphi_{1}(0) + \varphi_{2}(0) = 1$, this and $\beta_{i} \sigma_{i}^{2} = 2 b_{i}$ yield
\begin{align*}
 & \varphi_{1}(0) = \frac {\beta_{2} \sigma_{2}^{2} (e^{-\beta_{1} \ell_{1}} - 1) } {\beta_{1} \sigma_{1}^{2} + \beta_{2} \sigma_{2}^{2} (e^{-\beta_{1} \ell_{1}} - 1)} = \frac {b_{2} (e^{-\beta_{1} \ell_{1}} - 1) } {b_{1} + b_{2} (e^{-\beta_{1} \ell_{1}} - 1)} = d_{11},\\
 & \varphi_{2}(0) = \frac {\beta_{1} \sigma_{1}^{2}} {\beta_{1} \sigma_{1}^{2} + \beta_{2} \sigma_{2}^{2} (e^{-\beta_{1} \ell_{1}} - 1)} = \frac {b_{1}} {b_{1} + b_{2} (e^{-\beta_{1} \ell_{1}} - 1)} = d_{12}.
\end{align*}
This proves \eq{d11}. We next assume that $b_{1} = 0$, then it follows from \eq{phi5} and \eq{phi21} that
\begin{align*}
  \frac {\varphi_{1}(0)} {\varphi_{2}(0)} = \frac {-\beta_{2} \sigma_{2}^{2} \ell_{1}} {\sigma_{1}^{2}}.
\end{align*}
Similarly to the case for $b_{1} \not= 0$, this yields
\begin{align}
\label{eq:phi(0)-0}
 & \varphi_{1}(0) = \frac {-\beta_{2} \sigma_{2}^{2}  \ell_{1}} {\sigma_{1}^{2} - \beta_{2} \sigma_{2}^{2} \ell_{1}} = \frac {-2b_{2}  \ell_{1}} {\sigma_{1}^{2} - 2b_{2} \ell_{1}} = d_{01},\\
 & \varphi_{2}(0) = \frac {\sigma_{1}^{2}} {\sigma_{1}^{2} - \beta_{2} \sigma_{2}^{2} \ell_{1}} = \frac {\sigma_{1}^{2}} {\sigma_{1}^{2} - 2b_{2} \ell_{1}} = d_{02},
\end{align}
This proves \eq{d01}. Thus, the proof of \thr{k=2} is completed.

\subsection{Stationary distribution for general $k$}
\label{sec:general k}

We now derive the stationary distribution of the one-dimensional $k$-level SRBM for a general positive integer $k$. Recall the definition of $\beta_{j}$, and define $\eta_{j}$ as
\begin{align*}
  \beta_{j} = 2 b_{j}/\sigma_{j}^{2}, \; j \in \sr{N}_{k}, \qquad \eta_{0} = 1, \; \eta_{j} = \prod_{i=1}^{j} e^{\beta_{i}(\ell_{i} - \ell_{i-1}^{+})}, \; j \in \sr{N}_{k-1}, \; \eta_{k} = 0,
\end{align*}
where $x^{+} = 0 \vee x \equiv \max(0,x)$ for $x \in \dd{R}$. Also recall that the state space $S$ is partitioned to $S_{j}$ defined in \eq{partition} for $j \in \sr{N}_{k}$.

\begin{theorem}[The case for general $k \ge 2$]
\label{thr:general k}
The $Z(\cdot)$ of the one-dimensional $k$-level SRBM has a stationary distribution if and only if $b_{k} < 0$, equivalently, $\beta_{k} < 0$. Let $J = \{i \in \sr{N}_{k}; b_{i} = 0\}$, and assume that $b_{k} < 0$, then denote the stationary distribution of $Z(t)$ by $\nu$, then $\nu$ is unique and has a probability density function $h^{J}$ for  which is given below.\\
 (i) If $J = \emptyset$, that is, $b_{j} \not= 0$ for all $j \in \sr{N}_{k}$, then 
\begin{align}
\label{eq:Z-h}
  h^{\emptyset}(x) = h(x) \equiv \sum_{j=1}^{k} d_{j} h_{j}(x), \qquad x \ge 0,
\end{align}
where $h_{j}$ for $j \in \sr{N}_{k}$ are probability density functions defined as
\begin{align}
\label{eq:hj}
 & h_{j}(x) = \begin{cases}
 \ds \frac { \beta_{j} e^{\beta_{j} (x - \ell_{j-1}^{+})}} {e^{\beta_{j} (\ell_{j} - \ell_{j-1}^{+})} - 1}1(x \in S_{j}), \qquad & j \in \sr{N}_{k-1}, \\
 -\beta_{k} e^{\beta_{k} (x - \ell_{k-1})} 1(x \in S_{k}), & j = k, 
\end{cases}
\end{align}
and $d_{j}$ for $j \in \sr{N}_{k}$ are positive constants defined as
\begin{align}
\label{eq:dj}
 & d_{j} = 
 \frac {b_{j}^{-1}(\eta_{j} - \eta_{j-1})} {\sum_{i=1}^{k-1} b_{i}^{-1}(\eta_{i} - \eta_{i-1}) - b_{k}^{-1} \eta_{k-1}}, \qquad j \in \sr{N}_{k}.
\end{align}
(ii) If $J \not= \emptyset$, that is, $b_{i} = 0$ for some $i \in J$, then
\begin{align}
\label{eq:Z-hJ}
  h^{J}(x) = \sum_{j=1}^{k} d^{J}_{j} h^{J}_{j}(x), \qquad x \ge 0,
\end{align}
where $\ds h^{J}_{j}(x) = \lim_{b_{i} \to 0, i \in J} h_{j}(x)$ and $\ds d^{J}_{j} = \lim_{b_{i} \to 0, i \in J} d_{j}$ for $j \in \sr{N}_{k}$.
\end{theorem}

Before proving this theorem in \sectn{proof-general k}, we note that the density $h^{\emptyset}$ has a simple expression, which is further discussed in \sectn{weaker conditions}.

\begin{corollary}\rm
\label{cor:general k}
Under the assumptions of \thr{general k}, the probability density function $h^{\emptyset}$ of the stationary distribution of the $k$-level $SRBM$ on $\dd{R}_{+}$ when $b(x) \not= 0$ for all $x \ge 0$  is given by
\begin{align}
\label{eq:hx}
  h^{\emptyset}(x) = \frac {1}{C_{k} \sigma^{2}(x)} \exp\left(\int_{0}^{x} \frac {2b(y)}{\sigma^{2}(y)} dy\right), \qquad x \ge 0.
\end{align}
where
\begin{align}
\label{eq:Ck}
  C_{k} = \int_{0}^{\infty} \frac {1}{\sigma^{2}(x)} \exp\left(\int_{0}^{x} \frac {2b(y)}{\sigma^{2}(y)} dy\right) dx.
\end{align}
 \end{corollary}

\begin{proof}
Let $C = \sum_{i=1}^{k} b_{i}^{-1}(\eta_{i} - \eta_{i-1})$, which is finite, and write $\eta_{j}$ for $j \in \sr{N}_{k-1}$ as
\begin{align*}
  \eta_{j} = \exp\left(\sum_{i=1}^{j} \beta_{i} (\ell_{i} - \ell^{+}_{i-1}) \right) = \exp\left(\int_{0}^{\ell_{j}} \frac {2b(y)}{\sigma^{2}(y)} dy \right),
\end{align*}
Then, from (i) of \thr{general k}, we have, for $x \in [\ell_{j-1}^{+}, \ell_{j})$ with $j \le k-1$.
\begin{align}
\label{eq:dj-k-1}
  d_{j} h_{j}(x) & = \frac {1}{C b_{j}} (\eta_{j} - \eta_{j-1}) \frac {\beta_{j} e^{\beta_{j} (x - \ell_{j-1}^{+})}} {e^{\beta_{j} (\ell_{j} - \ell_{j-1}^{+})}-1} \nonumber\\
  & = \frac {1}{C b_{j}} (\eta_{j} - \eta_{j-1}) \frac {e^{\beta_{j} (x - \ell_{j-1}^{+})} \eta_{j-1}} {\eta_{j} - \eta_{j-1}} \beta_{j} = \frac 2{C \sigma^{2}(x)} \exp\left(\int_{0}^{x} \frac {2b(y)}{\sigma^{2}(y)} dy\right),
\end{align}
because $\beta_{j}/b_{j} = 2/\sigma_{j}^{2} = 2/\sigma^{2}(x)$ for $x \in [\ell_{j-1}, \ell_{j})$. Similarly, for $x \ge \ell_{k-1}$,
\begin{align}
\label{eq:dj-k}
  d_{k} h_{k}(x) = \frac {\eta_{k-1}} {Cb_{k}} (-\beta_{k}) e^{\beta_{k} (x - \ell_{k-1})} = \frac 2{C \sigma^{2}(x)} \exp\left(\int_{0}^{x} \frac {2b(y)}{\sigma^{2}(y)} dy\right).
\end{align}
Hence, putting $C_{k} = C/2$, we have \eq{hx}.
\end{proof}

\begin{remark}\rm
\label{rem:general k}
$d_{j}$ defined by \eq{dj} must be positive, which is easily checked. Nevertheless, it is interesting that their positivity is visible through \eq{dj-k-1} and \eq{dj-k} of this corollary.
\end{remark}

\subsection{Proof of \thr{general k}}
\label{sec:proof-general k}

Similar to the proof of \thr{k=2}, the first statement is immediate from \lem{stability}, and we can assume that $Z(\cdot)$ is a stationary process since $b_{k} < 0$. We also always assume that $\theta \le 0$. Define moment generating functions (mgf):
\begin{align*}
  \varphi_{j}(\theta) = \dd{E}[e^{\theta Z(0)} 1(Z(0) \in S_{j})], \qquad j \in \sr{N}_{k},
\end{align*}
which are obviously finite because $\theta \le 0$. Then, the mgf $\varphi(\theta)$ of $Z(0)$ is expressed as
\begin{align*}
  \varphi(\theta) = \sum_{j=1}^{k} \varphi_{j}(\theta), \qquad \theta \le 0,
\end{align*}
and $d_{j} = \varphi_{j}(0)$ for $j \in \sr{N}_{k}$. 

We first prove (i). In this case, let $\widehat{h}_{j}$ be the mgf of $h_{j}$ for $j \in \sr{N}_{k}$, then
\begin{align}
\label{eq:mgf-hj}
 & \widehat{h}_{j}(\theta) = \begin{cases}
\ds \frac {\beta_{j}}{\beta_{j} + \theta} \frac {e^{(\theta + \beta_{j})(\ell_{j} - \ell_{j-1}^{+})} - 1} {e^{\beta_{j}(\ell_{j} - \ell_{j-1}^{+})} - 1} e^{\beta_{j} \ell_{j}}, \quad & j \in \sr{N}_{k-1}\\
\ds \frac {\beta_{k}} {\beta_{k} + \theta} e^{\theta \ell_{k-1}}, & j = k.
\end{cases}
\end{align}
Hence, \eq{Z-h} is obtained if we show that, for $j \in \sr{N}_{k}$,
\begin{align}
\label{eq:phi-j}
 & \varphi_{j}(\theta)/\varphi_{j}(0) = \widehat{h}_{j}(\theta), \\
\label{eq:d-j}
 & \varphi_{j}(0) = d_{j}.
\end{align}

To prove \eq{phi-j} and \eq{d-j}, we use the following convex function $f_{j}$ with parameter $\theta \le 0$ as a test function for the generalized Ito formula similar to \eq{SIE-test1}.
\begin{align}
\label{eq:SIE-test-j}
  f_{j}(x) = e^{\theta x} 1(x \le \ell_{j}) + e^{\theta \ell_{j}}1(x > \ell_{j}). \qquad x \in \dd{R}, \; j \in \sr{N}_{k-1}.
\end{align}
Since $f_{j}'(\ell_{j}-) = \theta e^{\theta \ell_{j}}$ and $f_{j}'(\ell_{1}+) = 0$, it follows from \eq{Mf} that
\begin{align*}
  \mu_{f}(\{\ell_{j}\}) = \lim_{\varepsilon \downarrow 0} f'((\ell_{j}+\varepsilon) -) - f'(\ell_{j}-) = f'(\ell_{j}+) - f'(\ell_{j}-) = - \theta e^{\theta \ell_{j}},
\end{align*}
and, $f''(x) = \theta^{2} e^{\theta x}$ for $x < \ell_{j}$. Hence, similarly to \eq{G-Ito1}, the generalized Ito formula \eq{G-Ito} for $f = f_{j}$ becomes
\begin{align}
\label{eq:G-Ito-j}
 & f_{j}(Z(t)) = f_{j}(Z(0)) + \int_{0}^{t} \theta e^{\theta Z(u)} 1(0 \le Z(u) < \ell_{j}) \sigma(Z(u)) dW(u)  \nonumber\\
 & \quad + \int_{0}^{t} \Big(b(Z(u)) \theta + \frac 12 \sigma^{2}(Z(u)) \theta^{2}\Big) e^{\theta Z(u)} 1(0 \le Z(u) < \ell_{j}) du \nonumber\\
 & \quad  - \frac 12 \theta e^{\theta \ell_{j}} L_{\ell_{j}}(t) + \theta Y(t), \qquad t \ge 0, \; \theta \le 0, \; j \in \sr{N}_{k-1}.
\end{align}

Similarly to the proof of \thr{k=2}, we next apply Ito formula for test function $f(x) = e^{\theta x}$ to \eq{SIE-Z}, then we have \eq{G-Ito2} for $b(x)$ and $\sigma(x)$ which are defined by \eq{multi-level}. From \eq{G-Ito-j} and \eq{G-Ito2}, we will compute the stationary distribution of $Z(t)$.

We first consider this equation for $j=1$. In this case, \eq{G-Ito-j} becomes
\begin{align}
\label{eq:G-Ito-j1}
 & f_{1}(Z(t)) = f_{1}(Z(0)) + \int_{0}^{t} \theta e^{\theta Z(u)} 1(0 \le Z(u) < \ell_{1}) \sigma_{1} dW(u)  \nonumber\\
 & \quad + \int_{0}^{t} \Big(b_{1} \theta + \frac 12 \sigma_{1}^{2} \theta^{2}\Big) e^{\theta Z(u)} 1(0 \le Z(u) < \ell_{1}) du \nonumber\\
 & \quad  - \frac 12 \theta e^{\theta \ell_{1}} L_{\ell_{1}}(t) + \theta Y(t), \qquad t \ge 0, \; \theta \le 0.
\end{align}
Then, by the same arguments in the proof of \thr{k=2}, we have \eq{RBM-BAR1} and \eq{L-time3}, which imply
\begin{align}
\label{eq:phi1-general}
  \varphi_{1}(\theta) = \frac {e^{\theta \ell_{1}} - e^{-\beta_{1} \ell_{1}}}{\beta_{1} + \theta} \frac {1}{\sigma_{1}^{2}} \dd{E}[L_{\ell_{1}}(1)], \qquad \varphi_{1}(0) = \frac {1 - e^{-\beta_{1}\ell_{1}}}{\sigma_{1}^{2} \beta_{1}} \dd{E}[L_{\ell_{1}}(1)].
\end{align}
Hence, we have
\begin{align}
\label{eq:mgf-h1}
 & \varphi_{1}(\theta)/\varphi_{1}(0) = \frac {e^{\theta \ell_{1}} - e^{-\beta_{1} \ell_{1}}}{\beta_{1} + \theta} \frac {\beta_{1}} {1 - e^{-\beta_{1} \ell_{1}}} = \widehat{h}_{1}(\theta), \qquad \theta \le 0.
\end{align}
Thus, \eq{phi-j} is proved for $j=1$. We prove \eq{d-j} after \eq{phi-j} is proved for all $j \in \sr{N}_{k}$.

We next prove \eq{phi-j} for $j \in \{2,3,\ldots,k-1\}$. In this case, we use $f_{j}(Z(1)$ of \eq{G-Ito-j}. Take the difference $f_{j}(Z(1)) - f_{j-1}(Z(1))$ for each fixed $j$ and take the expectation under which $Z(\cdot)$ is stationary, then we have
\begin{align*}
  \sigma_{j}^{2}(\beta_{j} + \theta) \varphi_{j}(\theta) - e^{\theta \ell_{j}} \dd{E}[L_{\ell_{j}}(1)] + e^{\theta \ell_{j-1}} \dd{E}[L_{\ell_{j-1}}(1)] = 0, \qquad \theta \le 0,
\end{align*}
because $\beta_{j} = 2b_{j}/\sigma_{j}^{2}$. This yields
\begin{align}
\label{eq:RBM-BARj}
  \varphi_{j}(\theta) = & \frac {1}{\sigma_{j}^{2}(\beta_{j} + \theta)} \left(e^{(\theta + \beta_{j})\ell_{j}} - e^{(\theta + \beta_{j})\ell_{j-1}}\right) e^{-\beta_{j} \ell_{j}} \dd{E}[L_{\ell_{j}}(1)] \nonumber\\
  & \quad + \frac {e^{\theta \ell_{j-1}}}{\sigma_{j}^{2}(\beta_{j} + \theta)} \left( e^{-\beta_{j}(\ell_{j} - \ell_{j-1})} \dd{E}[L_{\ell_{j}}(1)] - \dd{E}[L_{\ell_{j-1}}(1)] \right).
\end{align}
Since $\varphi_{j}$ is the mgf of a measure on $[\ell_{j-1}, \ell_{j})$, we must have
\begin{align}
\label{eq:L-timej}
  \dd{E}[L_{\ell_{j-1}}(1)] = e^{-\beta_{j}(\ell_{j} - \ell_{j-1})} \dd{E}[L_{\ell_{j}}(1)], \qquad 2 \le j \le k-1.
\end{align}
Hence, \eq{RBM-BARj} becomes, for $j \in \{2,3,\ldots,k-1\}$,
\begin{align}
\label{eq:phij}
 & \varphi_{j}(\theta) = \frac {e^{(\theta + \beta_{j}) \ell_{j}} - e^{(\theta + \beta_{j})\ell_{j-1}}} {\sigma_{j}^{2}(\beta_{j} + \theta)} e^{-\beta_{j} \ell_{j}} \dd{E}[L_{\ell_{j}}(1)],\\
\label{eq:phij(0)}
 & \varphi_{j}(0) = \frac {1 - e^{-\beta_{j}(\ell_{j} - \ell_{{j-1}})}}{\sigma_{j}^{2} \beta_{j}} \dd{E}[L_{\ell_{j}}(1)].
\end{align}
Hence, we have \eq{phi-j} for $j=2,3,\ldots,k-1$.

We finally prove \eq{phi-j} for $j=k$. Similarly to the case for $k=2$ in the proof of \thr{k=2}, it follows from \eq{G-Ito2} that
\begin{align}
\label{eq:RBM-BARk+1}
  \frac 12 \sum_{i=1}^{k} \sigma_{i}^{2} \left(\beta_{i} + \theta \right) \varphi_{i}(\theta) + \dd{E}[Y(1)] = 0, \qquad \theta \le 0.
\end{align}
Similarly, from \eq{G-Ito-j} for $j=k-1$, we have
\begin{align}
\label{eq:RBM-BARk}
   \frac 12 \sum_{i=1}^{k-1} \sigma_{i}^{2} \left(\beta_{i} + \theta \right) \varphi_{i}(\theta) + \dd{E}[Y(1)]  - \frac 12 e^{\theta \ell_{k-1}} \dd{E}[L_{\ell_{k-1}}(1)]= 0, \qquad \theta \le 0.
\end{align}
Taking the difference of \eq{RBM-BARk+1} and \eq{RBM-BARk}, we have
\begin{align*}
  \sigma_{k}^{2} \left(\beta_{k} + \theta \right) \varphi_{k}(\theta) = - e^{\theta \ell_{k-1}} \dd{E}[L_{\ell_{k-1}}(1)],
\end{align*}
which yields
\begin{align}
\label{eq:mgf-h(k+1)}
  \varphi_{k}(\theta) = \frac {-1} {\sigma_{k}^{2} \left(\beta_{k} + \theta \right)} e^{\theta \ell_{k-1}} \dd{E}[L_{\ell_{k-1}}(1)], \qquad \varphi_{k}(0) = \frac {-1}{\sigma_{k}^{2} \beta_{k}} \dd{E}[L_{\ell_{k-1}}(1)].
\end{align}
Hence, we have \eq{phi-j} for $j=k$. Namely,
\begin{align*}
 \varphi_{k}(\theta)/\varphi_{k}(0) = \frac {\beta_{k}} {\beta_{k} + \theta} e^{\theta \ell_{k-1}} = \widehat{h}_{k}(\theta).
\end{align*}

It remains to prove \eq{d-j} for $j \in \sr{N}_{k}$. For this, we note that \eq{L-time3} is still valid, which is
\begin{align*}
  2 \dd{E}[Y(1)] = e^{-\beta_{1} \ell_{1}} \dd{E}[L_{\ell_{1}}(1)]= e^{-\beta_{1} (\ell_{1}-\ell_{0}^{+})} \dd{E}[L_{\ell_{1}}(1)].
\end{align*}
Hence, recalling that $\eta_{j} = \prod_{i=1}^{j} e^{\beta_{i}(\ell_{i} - \ell_{i-1}^{+})}$, \eq{L-timej} yields
\begin{align}
\label{eq:L-timej2}
  \dd{E}[L_{\ell_{j}}(1)] = e^{\beta_{j}(\ell_{j} - \ell_{j-1}^{+})} \dd{E}[L_{\ell_{j-1}}(1)] = 2 \dd{E}[Y(1)] \eta_{j}, \qquad j \in \sr{N}_{k-1}.
\end{align}

From \eq{phij(0)}, \eq{mgf-h(k+1)}, \eq{L-timej2} and the fact that $(e^{\beta_{j}(\ell_{j} - \ell_{j-1}^{+})} - 1) \eta_{j-1} = \eta_{j} - \eta_{j-1}$, we have
\begin{align}
\label{eq:phij(0)-2}
  \varphi_{j}(0) = \begin{cases}
\ds 2 \dd{E}[Y(1)] \frac {\eta_{j} - \eta_{j-1}}{\sigma_{j}^{2} \beta_{j}}, \quad & j \in \sr{N}_{k-1},  \medskip\\
\ds 2 \dd{E}[Y(1)] \frac {-1}{\sigma_{k}^{2} \beta_{k}} \eta_{k-1}, &   j = k.
\end{cases}
\end{align}
Since $\sum_{j=1}^{k} \varphi_{j}(0) = 1$, it follows from \eq{phij(0)-2} that
\begin{align}
\label{eq:normalize}
  \frac 1{2 \dd{E}[Y(1)]} & = \sum_{i=1}^{k} \frac {\eta_{i} - \eta_{i-1}}{\sigma_{i}^{2} \beta_{i}},
\end{align}
because $\eta_{k} = 0$. Substituting this into \eq{phij(0)-2} and using $\sigma_{i}^{2} \beta_{i} = 2 b_{i}$, we have \eq{d-j} for $j \in \sr{N}_{k}$ because $d_{j}$ is defined by \eq{dj}.

(ii) is proved for $k=2$ from (i) and (a) of \rem{k=2}. It is not hard to see that this observation (a) is also valid for any $b_{j}$ for $j \in \sr{N}_{k}$. Hence, (ii) can be proved also for $k \ge 2$ from (i).

\section{Proofs of preliminary lemmas}
\label{sec:preliminary}
\setnewcounter

\begin{figure}[h] %  figure placement: here, top, bottom, or page
  \centering
  \includegraphics[height=5.5cm]{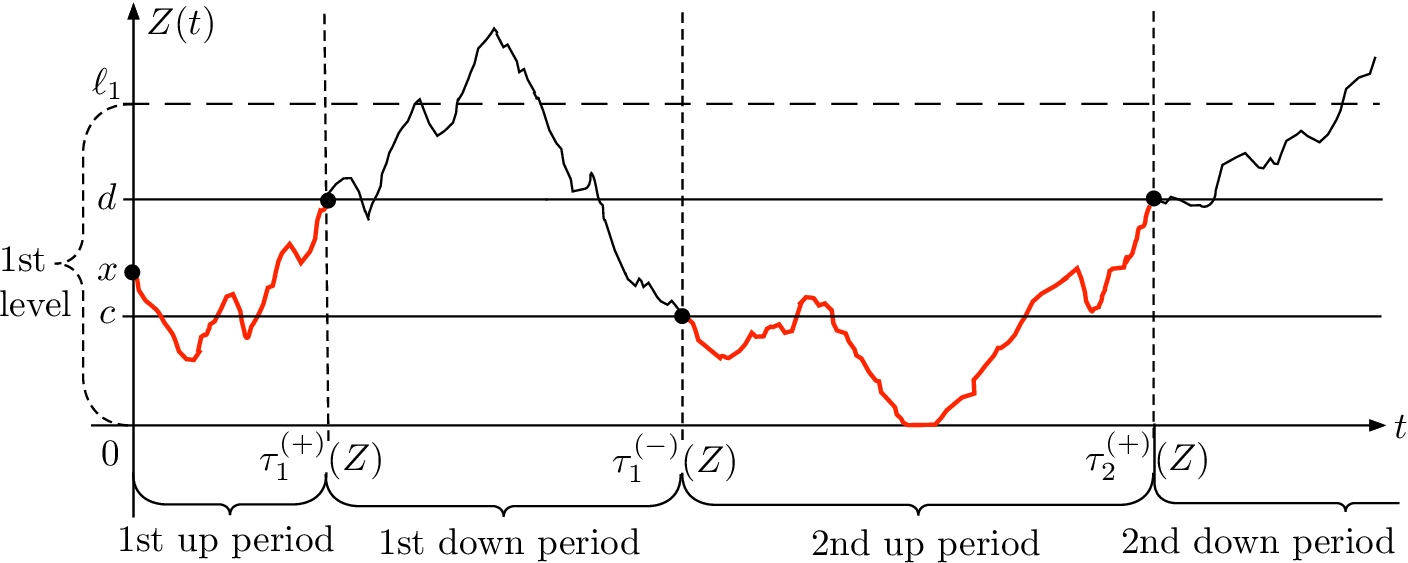} %\vspace{-1ex}
  \caption{Up and down level-crossing periods}
   \label{fig:spliting}
\end{figure}\vspace{-3ex}

\subsection{Proof of \lem{SIE-Z}}
\label{sec:existence}

Recall that \lem{SIE-Z} assumes the conditions of \eq{VD2+} and of the one-dimensional state-dependent SRBM with bounded drifts. Since $\ell_{1} > 0$, there are constants $c, d > 0$ such that $0 < c < d < \ell_{1}$. Using these constants, we construct the weak solution of $(Z(\cdot),W(\cdot))$ of \eq{SIE-Z}. The basic idea is to construct the sample path of $Z(\cdot)$ separately for disjoint time intervals, where, for the first interval, if $Z(0) < d$, then $Z(\cdot)$ stays there until it hits $d$ or, if $Z(0) \ge d$, then it stays there until it hits $c$, and, for the subsequent intervals, $Z(\cdot)$ starts below $c$ until hits $d > c$, which is called an up-crossing period, and those in which $Z(\cdot)$ starts start at $d$ or above it until hits $c < d$, which is called a down-crossing period. Namely, except for the first interval, the up-crossing period always starts at $c$, and the down-crossing period always starts at $d$ (see \fig{spliting}). In this construction, we also construct the filtration for which $(Z(\cdot),W(\cdot))$ is adapted.

Define $X_{1}(\cdot) \equiv \{X_{1}(t); t \ge 0\}$ as 
\begin{align}
 \label{eq:SIE-X1}
 X_{1}(t) = X_{1}(0) + \int_{0}^{t} \left(\sigma_{1} dW_{1}(u) + b_{1} du\right), \qquad t \ge 0,
\end{align}
and let $X_{2}(\cdot) \equiv \{X_{2}(t); t \ge 0\}$ be the solution of the following stochastic integral equation:
\begin{align}
\label{eq:SIE-X2}
  X_{2}(t) = X_{2}(0) & + \int_{0}^{t} \sigma(X_{2}(u)) dW_{2}(u) + \int_{0}^{t} b(X_{2}(u)) du, \qquad t \ge 0.
\end{align}
Note that the SIE \eq{SIE-X} is stochastically identical with the SIE \eq{SIE-X2}. Hence, as we discussed below \eq{SIE-X}, the SIE \eq{SIE-X2} has a unique weak solution under \cond{VD2}. Thus, the solution $X_{2}(\cdot)$ weakly exists because the assumptions of \lem{SIE-Z} imply \cond{VD2}. For this weak solution, we use the same notations for $X_{2}(\cdot)$, $W_{2}(\cdot)$ and stochastic basics $(\Omega,\sr{F},\dd{F},\dd{P})$ for convenience, where $\dd{F} = \{\sr{F}_{t}; t \ge 0\}$. Without loss of generality, we expand this stochastic basic which accommodates $X_{1}(\cdot)$, and have countable independent copies of $W_{i}(\cdot)$ and $X_{i}(\cdot)$ for $i=1,2$, which are denoted by $W_{n,i}(\cdot) \equiv \{W_{n,i}(t); t \ge 0\}$ and $X_{n,i}(\cdot) \equiv \{X_{n,i}(t); t \ge 0\}$ for $n=1,2,\ldots$.

We first construct the weak solution $Z(\cdot)$ of \eq{SIE-Z} when $Z(0) = x < d$, using $W_{n,i}(\cdot)$ and $X_{n,i}(\cdot)$. For this construction, we introduce up and down crossing times for a given real-valued semi-martingale $V(\cdot) \equiv \{V(t); t \ge 0\}$. Denote the $n$-th up-crossing time at $d$ from below by $\tau^{(+)}_{d,n}(V)$, and denote the down-crossing time at $c \; (< d)$ from above by $\tau^{(-)}_{c,n}(V)$. Namely, for $n \ge 1$ and $0 < c < d < \ell_{1}$,
\begin{align*}
  \tau^{(+)}_{d,n}(V) = \inf\{ u > \tau^{(-)}_{c,n-1}(V); V(u) \ge d\}, \quad \tau^{(-)}_{c,n}(V) = \inf\{ u > \tau^{(+)}_{d,n}(V); V(u) < c\},
\end{align*}
where $\tau^{(-)}_{c,0}(V) = 0$. Note that $\tau^{(+)}_{d,n}(Z)$ and $\tau^{(-)}_{c,n}(Z)$ may be infinite with positive probabilities. In this case, there is no further splitting, which causes no problem in constructing the sample path of $Z(\cdot)$ because such a sample path is already defined for all $t \ge 0$. After the weak solution is obtained, we will see that $\dd{P}_{y}[\tau^{(+)}_{d,n}(Z) < \infty] = 1$ for $y \in [0,d)$ by \lem{irreducible}, but $\tau^{(-)}_{c,n}(Z)$ may be infinite with a positive probability.

We now inductively construct $Z_{n}(t)\equiv \{Z_{n}(t); t \ge 0\}$ for $n=1,2,\ldots$, where the construction below is stopped when $\tau^{(-)}_{c,n}(Z_{n})$ diverges. For $n=1$, we denote the independent copy of $X_{1}(\cdot)$ with $X_{1}(0)=x < d$ by $X_{11}(\cdot) \equiv \{X_{11}(t); t \ge 0\}$, and define $Z_{11}(t)$ as
\begin{align}
\label{eq:ZX11}
 & Z_{11}(t) = X_{11}(t) + \sup_{u \in [0,t]} (-X_{11}(u))^{+}, \qquad  t \ge 0,
\end{align}
then it is well known that $Z_{11}(\cdot)$ is the unique solution of the stochastic integral equation:
\begin{align}
\label{eq:Z10}
  & Z_{11}(t) = Z_{11}(0) + \int_{0}^{t} d X_{11}(u) + Y_{11}(t), \qquad t \ge 0,
\end{align}
where $Y_{11}(t)$ is nondecreasing and $\int_{0}^{t} 1(Z_{11}(u) > 0) dY_{11}(u) = 0$ for $t \ge 0$. Furthermore, for $X_{11}(0) = Z_{11}(0)$, $Y_{11}(t) = \sup_{u \in [0,t]} (-X_{11}(u))^{+}$ (e.g., see \cite{KrukLehoRama2007}). Since $Z_{11}(0) = X_{11}(0) = x < d$ and $X_{11}(t) \le Z_{11}(t) \le d < \ell_{1}$ for $t \in [0,\tau^{(+)}_{d,1}(Z_{11})]$, \eq{Z10} can be written as
\begin{align}
\label{eq:Z11}
  & Z_{11}(t) = Z_{11}(0)  \nonumber\\
  & \quad + \int_{0}^{t} \left(\sigma(Z_{11}(u)) d W_{11}(u) + b(Z_{11}(u)) du\right) + Y_{11}(t), \qquad t \in [0,\tau^{(+)}_{d,1}(Z_{11})).
\end{align}
We next denote the independent copy of $X_{2}(\cdot)$ with $X_{2}(0)=d$ by $X_{12}(\cdot) \equiv \{X_{12}(t); t \ge 0\}$, and define
\begin{align}
\label{eq:ZX12}
  Z_{12}(t) = X_{12}\left(t - \tau^{(+)}_{d,1}(Z_{11})\right), \qquad t \ge \tau^{(+)}_{d,1}(Z_{11}),
\end{align}
then we have, for $t \in [\tau^{(+)}_{d,1}(Z_{11}), \tau^{(-)}_{c,1}(Z_{12}))$,
\begin{align}
\label{eq:Z12}
  Z_{12}(t) = d + \int_{\tau^{(+)}_{d,1}(Z_{11})}^{t} \left(\sigma(Z_{12}(u)) d W_{12}(u) + b(Z_{12}(u)) du\right),
\end{align}
where recall that $X_{12}(0) = d$. Define
\begin{align*}
  Z_{1}(t) & = Z_{11}(\tau^{(+)}_{d,1}(Z_{11}) \wedge t) + (Z_{12}(t) - d) 1(\tau^{(+)}_{d,1}(Z_{11}) < t), \qquad t \ge 0.
\end{align*}
then $Z_{11}(t)$ is stochastically identical with $Z_{12}(t)$ for $t \in [\tau^{(+)}_{d,1}(Z_{11}), \tau^{(+)}_{\ell_{1},1}(Z_{11}) \wedge \tau^{(+)}_{\ell_{1},1}(Z_{12}))$. Hence, it follows from \eq{Z11} and \eq{Z12} that, for $t \in [0, \tau^{(-)}_{c,1}(Z_{1}))$,
\begin{align}
\label{eq:Z1}
  Z_{1}(t) = Z_{1}(0) + \int_{0}^{t} \left(\sigma(Z_{1}(u)) d \widetilde{W}_{1}(u) + b(Z_{1}(u)) du\right) + Y_{1}(t),
\end{align}
where $Y_{1}(t) = Y_{11}(\tau^{(+)}_{d,1}(Z_{11}) \wedge t)$ and
\begin{align}
\label{eq:W1}
  \widetilde{W}_{1}(t) = W_{11}(t) 1(t < \tau^{(+)}_{d,1}(Z_{11})) + W_{12}(t) 1(t \ge \tau^{(+)}_{d,1}(Z_{11})).
\end{align}

We repeat the same procedure to inductively define $X_{n,i}(t)$ with $X_{n,i}(0) = c \, 1(i=1) + d \, 1(i=2)$ and $Z_{n,i}(t)$ for $i=1,2$ together with $\tau^{(+)}_{d,n}(Z_{n1})$ and $\tau^{(-)}_{c,n}(Z_{n2})$ for $n \ge 2$ by
\begin{align*}
 & X_{n1}(t) = c + \int_{\tau^{(-)}_{c,n-1}(Z_{n-1})}^{t} \left(\sigma_{1} d W_{n1}(u) + b_{1} du\right), \qquad t \ge \tau^{(-)}_{c,n-1}(Z_{n-1})\\
 & Z_{n1}(t) = X_{n1}(t) + \sup_{u \in [\tau^{(-)}_{c,n-1}(Z_{n-1}),t]} (-X_{n1}(u))^{+}, \qquad t \ge \tau^{(-)}_{c,n-1}(Z_{n-1})\\
 & Z_{n2}(t) = X_{n2}\left(t - \tau^{(+)}_{d,n}(Z_{n1})\right), \qquad t > \tau^{(+)}_{d,n}(Z_{n1}), \; X_{n2}(0) = d,
\end{align*}
as long as $\tau^{(-)}_{c,n-1}(Z_{n-1}) < \infty$, and define
\begin{align*}
  Z_{n}(t) & = Z_{n-1}(\tau^{(-)}_{c,n-1}(Z_{n-1}) \wedge t) + (Z_{n1}(\tau^{(+)}_{d,n}(Z_{n1}) \wedge t)  - c) 1(\tau^{(-)}_{c,n-1}(Z_{n-1}) < t) \nonumber\\
  & \quad + (Z_{n2}(t) - d) 1(\tau^{(+)}_{d,n}(Z_{n1}) < t),
\end{align*}
then we have, for $t \in [0, \tau^{(-)}_{c,n}(Z_{n2}))$,
\begin{align}
\label{eq:Zn}
  Z_{n}(t) = Z_{n}(0) + \int_{0}^{t} \left(\sigma(Z_{n}(u)) d \widetilde{W}_{n}(u) + b(Z_{n}(u)) du\right) + Y_{n}(t),
\end{align}
where
\begin{align*}
 & Y_{n}(t) = Y_{n-1}(\tau^{(-)}_{c,n-1}(Z_{n-1}) \wedge t) + \sup_{u \in [\tau^{(-)}_{c,n-1}(Z_{n-1}),\tau^{(+)}_{d,n}(Z_{n1}) \wedge t]} (-X_{n}(u))^{+},\\
 & \widetilde{W}_{n}(t) = \widetilde{W}_{n-1}(t) 1(t \le \tau^{(+)}_{d,n-1}(Z_{(n-1)1}) + \widetilde{W}_{(n-1)2}(t) 1(t > \tau^{(+)}_{d,1}(Z_{(n-1)1})).
\end{align*}

From \eq{Zn}, we can see that $Z_{n}(\cdot) \equiv \{Z_{n}(t); 0 \le t < \tau^{(-)}_{c,n}(Z_{n2})\}$ is the solution of \eq{SIE-Z} for $t < \tau^{(-)}_{c,n}(Z_{n2})$. Furthermore, $Z_{n}(t) = Z_{n+1}(t)$ for $0 \le t < \tau^{(-)}_{c,n}(Z_{n2})$. From this observation, we define $Z(\cdot)$ by $Z(0) = x$ and
\begin{align}
\label{eq:Zt}
 & Z(t) = Z(0) + \sum_{n=1}^{\infty} Z_{n}(t) 1(\tau^{(-)}_{c,n-1}(Z_{n-1}) \le t < \tau^{(-)}_{c,n}(Z_{n})),\\
\label{eq:Yt}
 & Y(t) = \sum_{n=1}^{\infty} Y_{n}(t) 1(\tau^{(-)}_{c,n-1}(Z_{n-1}) \le t < \tau^{(-)}_{c,n}(Z_{n})),\\
\label{eq:Wt}
 & W(t) = \sum_{n=1}^{\infty} \widetilde{W}_{n}(t) 1(\tau^{(-)}_{c,n-1}(Z_{n-1}) \le t < \tau^{(-)}_{c,n}(Z_{n})),
 \qquad t \ge 0,
\end{align}
where $\tau^{(-)}_{c,0}(Z_{0}) = 0$, then $Z(\cdot)$ is the solution of \eq{SIE-Z} for $t < \tau^{(-)}_{c,n}(Z_{n2})$ if $\tau^{(-)}_{c,n-1}(Z_{n-1}) < \infty$. Otherwise, if $\tau^{(-)}_{c,n-1}(Z_{n-1}) = \infty$ and $\tau^{(-)}_{c,m}(Z_{n-1}) < \infty$ for $m = 1,2, \ldots, n-2$, then we stop the procedure by the $(n-1)$-th step.

Up to now, we have assumed that $Z(0) = Z_{1}(0) = Z_{11}(0) = x < d$. If this $x$ is net less than $d$, then we start with $Z_{12}(\cdot)$ of \eq{ZX12} with $Z_{12}(0) = X_{12}(0) = x \ge d$, and replace $Z_{11}(\cdot)$ of \eq{ZX11} by
\begin{align*}
 & Z_{11}(t) = X_{11}(t) + \sup_{\tau^{(-)}_{c,1}(Z_{12}) < u \le t} (-X_{11}(u))^{+}, \qquad  t \ge 0.
\end{align*}
Then, define $Z_{1}(\cdot)$ as
\begin{align*}
  Z_{1}(t) & = Z_{12}(\tau^{(-)}_{c,1}(Z_{12}) \wedge t) + (Z_{11}(t) - c) 1(\tau^{(-)}_{c,1}(Z_{12}) \le t), \qquad t \ge 0,
\end{align*}
where $\tau^{(-)}_{c,1}(Z_{12}) < \tau^{(+)}_{d,1}(Z_{11})$ because the order of $Z_{11}(\cdot)$ and $Z_{12}(\cdot)$ is swapped. Similarly to the previous case that $x < d$, we repeat this procedure to inductively define $Z_{n}(\cdot)$ for $n \ge 2$, then we can defined $Z(\cdot)$ and $Y(\cdot)$ similarly to \eq{Zt} and \eq{Yt}.

Hence, $Z(\cdot)$ of \eq{Zt} is the solution of \eq{SIE-Z} if we show that there is some $n \ge 1$ for each $t > 0$ such that $t < \tau^{(-)}_{c,n}(Z)$. This condition is equivalent to $\sup_{n \ge 1} \tau^{(-)}_{c,n}(Z) = \infty$ almost surely. To see this, assume that $\tau^{(-)}_{c,n}(Z) < \infty$ for all $n \ge 1$, then let $J_{n} = \tau^{(-)}_{c,n}(Z) - \tau^{(-)}_{c,n-1}(Z)$ for $n \ge 1$, then $\{J_{n}; n \ge 2\}$ is a sequence of $i.i.d.$ positive valued random variables. Hence, we have
\begin{align*}
  \lim_{n \to \infty} \tau^{(-)}_{c,n}(Z) \ge \lim_{n \to \infty} \sum_{m=2}^{n} J_{m} = \infty, \qquad a.s.,
\end{align*}
and therefore $Z(t)$ is well defined for all $t \ge 0$. Otherwise, if $\tau^{(-)}_{c,n}(Z) = \infty$ for some $n \ge 1$, then we stop the procedure by the $n$-th step.

Thus, we have constructed the solution $Z(\cdot)$ of \eq{SIE-Z}. Note that the probability distribution of this solution does not depend on the choice of $c, d$ as long as $0 < c < d < \ell_{1}$ because of the independent increment property of the Brownian motion. Furthermore, this $Z(\cdot)$ is a strong Markov process because $Z_{n,1}(\cdot)$ and $Z_{n,2}(\cdot)$ are strong Markov processes (e.g. see (8.12) of \cite{ChunWill1990}, Theorem 21.11 of \cite{Kall2001}, Theorem 17.23 and Remark 17.2.4 of \cite{CoheElli2015}) and $Z(\cdot)$ is obtained by continuously connecting their sample paths using stopping times. Thus, the $Z(\cdot)$ is the weak solution of \eq{SIE-Z} which is strong Markov.

It remains to prove the weak uniqueness of the solution $Z(\cdot)$. This is immediate from the construction of $Z(\cdot)$. Namely, suppose that $\widetilde{Z}(\cdot)$ is the solution of \eq{SIE-Z} with $\widetilde{Z}(0) = x$ for given $x \ge 0$. Assume that $x < d$, then the process $\{\widetilde{Z}(t); 0 \le t < \tau^{(+)}_{d,1}(\widetilde{Z})\}$ with $\widetilde{Z}(0)=x < d < \ell_{1}$ is stochastically identical with $\{Z_{11}(t); 0 \le t < \tau^{(+)}_{d,1}(Z)\}$ with $Z_{11}(0)=x$, which is the unique solution of \eq{Z10}, while the process $\{\widetilde{Z}(t); \tau^{(+)}_{d,1}(\widetilde{Z}) < t \le \tau^{(-)}_{c,1}(\widetilde{Z})\}$ must be stochastically identical with $\{X_{12}(t); 0 \le t < \tau^{(-)}_{c,1}(Z)\}$ with $X_{12}(0)=d$, which is the unique weak solution of \eq{SIE-X2}. Similarly, we can see such stochastic equivalences in the subsequent periods for $\widetilde{Z}(0)=x < d$. On the other hand, if $\widetilde{Z}(0) = x \ge d$, then similar equivalences are obtained. Hence, $\widetilde{Z}(\cdot)$ and $Z(\cdot)$ have the same distribution for each fixed initial state $x \ge 0$. Thus, the $Z(\cdot)$ is a unique weak solution, and the proof of \lem{SIE-Z} is completed.

\begin{remark}
\label{rem:existence}
From an analogy to the reflecting Brownian motion on the half line $[0,\infty)$, it may be questioned whether the solution $Z(\cdot)$ of \eq{SIE-Z} can be directly obtained from the weak solution $X(\cdot)$ of \eq{SIE-X} by its absolute value, that is by $|X|(\cdot) \equiv \{|X(t)|; t \ge 0\}$. This question is affirmatively answered under \cond{VD2} by \citet{AtarCastReim2022}. It may be interesting to see how they prove (i) of \lem{SIE-Z0}, so we explain it below.

Recall that the solution $X(\cdot)$ of the SIE \eq{SIE-X} weakly exists under \cond{VD2}. If $|X|(\cdot)$ is the solution $Z(\cdot)$ of the stochastic integral equation \eq{SIE-Z}, then we must have
\begin{align}
\label{eq:SIE-|X|}
  |X|(t) = |X|(0) + \int_{0}^{t} (\sigma(|X|(u)) dW(u) + b(|X|(u)) du) + Y(t), \qquad t \ge 0.
\end{align}
On the other hand, from Tanaka formula \eq{Tanaka} for $a = 0$, we have
\begin{align}
\label{eq:Tanaka-X}
 & |X|(t) - |X|(0) = \int_{0}^{t} \sgn (X(u)) dX(u) + L_{0}(t) \nonumber\\
 & \quad = \int_{0}^{t} \sgn (X(u)) \left(\sigma(X(u)) dW(u) + b(X(u)) du\right)+ L_{0}(t), \qquad t \ge 0.
\end{align}
Hence, letting $Y(\cdot) = L_{0}(\cdot)$, \eq{SIE-|X|} is stochastically identical with \eq{Tanaka-X} if
\begin{align}
\label{eq:sgn-VD}
  \sigma(x) = \sigma(|x|), \qquad b(x) = \sgn(x) b(|x|), \qquad x \in \dd{R},
\end{align}
and if $W(\cdot)$ is replaced by $\widetilde{W}(\cdot) \equiv \{\sgn(X(t)) W(t); t \ge 0\}$. Since the stochastic integral in \eq{SIE-Z} does not depend on $\sigma(x)$ and $b(x)$ for $x < 0$, \eq{sgn-VD} does not cause any problem for \eq{SIE-Z}.
\end{remark}

\subsection{Proof of \lem{irreducible}}
\label{sec:irreducible}

Recall the definition of $\tau_{a} = \tau_{B}$ for $B = \{a\}$ (see \eq{tau-B}). We first prove that
\begin{align}
\label{eq:E-tau-a finite}
 & \dd{E}_{x}[\tau_{a}] < \infty, \qquad 0 \le x < a,\\
\label{eq:P-tau-a positive}
 & \dd{P}_{x}[\tau_{a} < \infty] > 0, \qquad 0 \le a < x,
\end{align}
Since $Z(\tau_{a} \wedge t) \le a$, $\dd{E}_{x}[e^{\theta Z(\tau_{a} \wedge t)}] < \infty$ for $\theta \in \dd{R}$. Hence, substituting the stopping time $\tau_{a} \wedge t$ into $t$ of the generalize Ito formula \eq{G-Ito2} for test function $f(x) = e^{\theta x}$ and taking the expectation under $\dd{P}_{x}$, we have, for $x <  a$ and $\theta \in \dd{R}$,
\begin{align}
\label{eq:G-Ito3}
 & \dd{E}_{x}[e^{\theta Z(\tau_{a} \wedge t)}] = e^{\theta x} + \dd{E}_{x} \left[\int_{0}^{\tau_{a} \wedge t} \gamma(Z(u),\theta) e^{\theta Z(u)} du\right] + \theta \dd{E}_{x}[Y(\tau_{a} \wedge t)], 
\end{align}
where $\gamma(x,\theta) = b(x) \theta + \frac 12 \sigma^{2}(x) \theta^{2}$. Note that f, for each $\varepsilon > 0$, $\gamma(x,\theta) \ge \varepsilon$ if
\begin{align}
\label{eq:theta-bound}
   \theta \ge \frac {-b(x)}{\sigma^{2}(x)} + \sqrt{\left(\frac{b(x)}{\sigma^{2}(x)}\right)^{2} + \frac {2\varepsilon} {\sigma^{2}(x)}} \; \mbox{ or } \; \theta \le \frac {-b(x)}{\sigma^{2}(x)} - \sqrt{\left(\frac{b(x)}{\sigma^{2}(x)}\right)^{2} + \frac {2\varepsilon} {\sigma^{2}(x)}}.
\end{align}
Recall that $\beta_{i} = 2b_{i}/\sigma^{2}_{i}$, and introduce the following notations.
\begin{align*}
  |\beta|_{\max} = \max_{i \in \sr{N}_{k}} |\beta_{i}|, \quad |\beta|_{\min} = \min_{i \in \sr{N}_{k}} |\beta_{i}|, \quad \sigma^{2}_{\max} = \max_{i \in \sr{N}_{k}} \sigma^{2}_{i}, \quad \sigma^{2}_{\min} = \min_{i \in \sr{N}_{k}} \sigma^{2}_{i}.
\end{align*}
Then, $|\beta|_{\max} < \infty$, $\sigma^{2}_{\max} < \infty$ and $\sigma^{2}_{\min} > 0$ by \cond{VD2}, which is assumed in \lem{irreducible}. Hence, for each $\varepsilon > 0$, $\gamma(x,\theta) \ge \varepsilon$ for $\theta \ge \frac 12\left(|\beta|_{\max} + \sqrt{|\beta|^{2}_{\max} + 8\varepsilon/\sigma^{2}_{\min}}\right)$ and $x \ge 0$. For this $\theta$, it follows from \eq{G-Ito3} that
\begin{align*}
  & e^{\theta a} - e^{\theta x} \ge \varepsilon \dd{E}_{x}\left[\int_{0}^{\tau_{a} \wedge t} e^{\theta Z(u)} du\right] \ge \varepsilon e^{\theta x}\dd{E}_{x}[\tau_{a} \wedge t], \qquad t \ge 0,
\end{align*}
because $\theta > 0$ and $e^{\theta Z(u)} \ge 1$ for $u \in [0,\tau_{a} \wedge t]$. This proves \eq{E-tau-a finite} because we have
\begin{align}
\label{eq:tau-bound1}
  \dd{E}_{x}[\tau_{a} \wedge t] \le (e^{\theta a} - e^{\theta x})/\varepsilon < \infty, \qquad x < a.
\end{align}

We next consider the case for $x > a > 0$. Similarly to the previous case but for $\theta < 0$, from \eq{theta-bound}, we have $\gamma(x,\theta) > \varepsilon$ for $x > a$ and $\varepsilon > 0$ if $\theta$ satisfies
\begin{align}
\label{eq:theta-}
  \theta \le - \frac 12\left(|\beta|_{\max} + \sqrt{|\beta|^{2}_{\max} + 8\varepsilon/\sigma^{2}_{\min}}\right) < 0.
\end{align}
Since $Y(t) = 0$ for $t \le \tau_{a}$ because $Z(0)=x > a$, we have, from \eq{G-Ito3}, for $\theta$ satisfying \eq{theta-},
\begin{align}
\label{eq:x>a>0}
  & \quad \dd{E}_{x}[e^{\theta Z(\tau_{a} \wedge t)}] \ge e^{\theta x} + \varepsilon \int_{0}^{t} \dd{E}_{x} [e^{\theta Z(u)} 1(u \le \tau_{a})]du, \qquad t \ge 0.
\end{align}
Assume that $\dd{P}_{x}(\tau_{a} = \infty) = 1$, then $\dd{P}_{x}[t > \tau_{a}] = 0$, so we have, from \eq{x>a>0}, 
\begin{align*}
  \dd{E}_{x}[e^{\theta Z(t)}1(t \le \tau_{a})] = \dd{E}_{x}[e^{\theta Z(\tau_{a} \wedge t)}] \ge e^{\theta x} + \varepsilon \int_{0}^{t} \dd{E}_{x} [e^{\theta Z(u)} 1(u \le \tau_{a})]du, \qquad t \ge 0.
\end{align*}
Denote $\dd{E}_{x}[e^{\theta Z(u)}1(u \le \tau_{a})]$ by $g(u)$. Then, after elementary manipulation, this yields
\begin{align*}
  \frac {d}{dt} \left(e^{-\varepsilon t} \int_{0}^{t} g(u) du\right) \ge e^{\theta x} e^{-\varepsilon t},
\end{align*}
and therefore, by integrating both sides of this inequality, we have
\begin{align*}
 e^{\theta a} \ge \frac 1t \int_{0}^{t} g(u) du \ge e^{\theta x} \frac {e^{\varepsilon t} - 1} {\varepsilon t},
\end{align*}
because $g(u) = \dd{E}_{x}[e^{\theta Z(u)}1(t \le \tau_{a})] \le e^{\theta a}$ for $\theta < 0$. Letting $t \to \infty$ in this inequality, we have a contradiction because its right-hand side diverges. Hence, we have \eq{P-tau-a positive}. We finally consider the case $0 = a < x$. If $\dd{P}_{x}[Y(\tau_{0}) = 0] = 1$, then \eq{x>a>0} holds, and the arguments below it works, which proves \eq{P-tau-a positive}. Otherwise, if $\dd{P}_{x}(Y(\tau_{0}) = 0) < 1$, that is, $\dd{P}_{x}(Y(\tau_{0}) > 0) > 0$, then $\dd{P}_{x}[\tau_{0} < \infty] > 0$ because of the definition of $Y(\cdot)$. Hence, we again have \eq{P-tau-a positive} for $a=0$.

We finally check Harris irreducible condition (see \eq{Harris1}). For this, let $\tau = \tau_{0} \wedge \tau_{\ell_{0}}$, then $\{Z(t); t \in (0,\tau)\}$ is stochastically identical with $\{X(t); t \in (0,\tau)\}$, where $X(t) \equiv X(0) + b_{1} t + \sigma_{1} W(t)$. Then, from Tanaka's formula \eq{Tanaka-} for $Z(\cdot)$, if $Z(0) = y \in (x, \ell_{1})$, 
\begin{align*}
  \frac 12 L_{x}(\tau \wedge t) & = (Z(\tau \wedge t)-x)^{-}\\
 & \quad + b_{1} \int_{0}^{\tau \wedge t} 1(Z(u) \le x) du + \sigma_{1} \int_{0}^{\tau \wedge t} 1(Z(u) \le x) dW(u).
\end{align*}
Hence, if $b_{1} \ge 0$, then
\begin{align}
\label{eq:La>0 1}
  \dd{E}_{y}[L_{x}(t)] \ge 2\dd{E}_{y}[(x - Z(\tau \wedge t)1(x > Z(\tau \wedge t)] > 0, \qquad t > 0.
\end{align}
Similarly, from \eq{Tanaka+} for $X(\cdot) = Z(\cdot)$, if $b_{1} < 0$, then, for $y \in (0,x) \subset (0,\ell_{1})$, 
\begin{align}
\label{eq:La>0 2}
  \dd{E}_{y}[L_{x}(t)] & \ge \dd{E}_{y}[L_{x}(\tau \wedge t)] \nonumber\\
 & \ge 2\dd{E}_{y}[(Z(\tau \wedge t)-x)^{+}] - 2b_{1} \dd{E}_{y}\left[\int_{0}^{\tau \wedge t} 1(Z(u) > x) du\right] > 0.
\end{align}

Assume that $b_{1} \ge 0$, then we choose $y \in (x, \ell_{1})$ and the Lebesque measure on $[0,y]$ for $\psi$. Then, it follows from \eq{La>0 1} and \eq{L-time1} with $g = 1_{B}$ that $\psi(B) > 0$ for $B \in \sr{B}(\dd{R}_{+})$ implies
\begin{align*}
  \dd{E}_{y}\left[\int_{0}^{t} 1_{B}(Z(u)) \sigma_{1}^{2} du\right] \ge \dd{E}_{y} \left[\int_{0}^{\infty} 1_{B}(x) L_{x}(t) \psi(dx)\right] > 0.
\end{align*}
Since $Z(\cdot)$ hits state $y \in (0,\ell_{1})$ from any state in $S$ with positive probability, this inequality implies the Harris irreducibility condition \eq{Harris1}. Similarly, this condition is proved for $b_{1} < 0$ using \eq{La>0 2} and the Lebesgue measure on $[y, \ell_{1}]$ for $\psi$. Thus, the proof of \lem{irreducible} is completed.

\subsection{Proof of \lem{stability}}
\label{sec:stability}

Obviously, $b_{*} < 0$ is necessary for $Z(\cdot)$ to have a stationary distribution because $Z(t)$ $a.s.$ diverges if $b_{*} > 0$ by the strong law of large numbers and \lem{irreducible} while $Z(\cdot)$ is null recurrent if $b_{*} = 0$.

Conversely, assume that $b_{*} < 0$. We note the following fact which is partially a counter part of \eq{E-tau-a finite}.

\begin{lemma}
\label{lem:Harris}
If $b_{*} < 0$ , then
\begin{align}
\label{eq:Harris5}
  \dd{E}_{x}[\tau_{a}] < \infty, \qquad \ell_{*} \le a < x.
\end{align}
\end{lemma}

\begin{proof}
Assume that $\ell_{*} \le a < x$, and let $X(t) \equiv x + b_{*}t + \sigma_{*} W(t)$ for $t \ge 0$. Since $\ell_{*} < x$, $\{Z(t); 0 \le t \le \tau^{X}_{\ell_{*}}\}$ under $\dd{P}_{x}$ has the same distribution as $\{X(t); 0 \le t \le \tau^{X}_{\ell_{*}}\}$, where $\tau^{X}_{y} = \inf\{u \ge 0; X(u) = y\}$ for $y \ge \ell_{*}$. Hence, applying the optional sampling theorem to the martingale $X(t) - x - b_{*} t$ for stopping time $\tau^{X}_{a} \wedge t$, we have
\begin{align*}
  \dd{E}_{x}[X(\tau^{X}_{a} \wedge t) - x - b_{*} (\tau^{X}_{a} \wedge t)] = 0, \qquad t \ge 0.
\end{align*}
Since $X(t) \to - \infty$ as $t \to \infty$ w.p.1 by strong law of large numbers, letting $t \to \infty$ in this equation yields $b_{*} \dd{E}_{x}[\tau^{X}_{a}] = a - x$. Hence, we have
\begin{align}
\label{eq:tau-bound2}
  \dd{E}_{x}[\tau_{a}] = \dd{E}_{x}[\tau^{X}_{a}] = \frac {-1}{b_{*}} (x - a), \qquad \ell_{*} \le a < x,
\end{align}
This proves \eq{Harris5}.
\end{proof}

We return to the proof of \lem{stability}. For $n \ge 1$ and $x, y > \ell_{*}$ such that $x < y$, inductively define $S_{n}, T_{n}$ as
\begin{align*}
  S_{n} = \inf\{t > T_{n-1}; Z(t) = y\}, \qquad T_{n} = \inf\{t > S_{n}; Z(t) = x\},
\end{align*}
where $T_{0} = 0$. Because $\ell_{*} < x < y$, we have, from \eq{E-tau-a finite} and \eq{Harris5}, 
\begin{align*}
  0 < \dd{E}_{x}[\tau_{y}] < \dd{E}_{x}[T_{1}] \le \dd{E}_{x}[\tau_{y}] + \dd{E}_{y}[\tau_{x}] < \infty.
\end{align*}
Hence, $Z(\cdot)$ is a regenerative process with regeneration cycles $\{T_{n}; n \ge 1\}$ because the sequence of $\{Z(t); T_{n-1} \le t < T_{n}\}$ for $n \ge 1$ is $i.i.d.$ by its strong Markov property. Hence, $Z(\cdot)$ has the stationary probability measure $\pi$ given by
\begin{align}
\label{eq:invariant-m}
  \pi(B) = \frac 1{\dd{E}_{x}(T_{1})} \dd{E}_{x}\left[ \int_{0}^{T_{1}} 1(Z(u) \in B) du \right], \qquad B \in \sr{B}(\dd{R}_{+}).
\end{align}
Thus, $Z(\cdot)$ is positive recurrent.

\section{Concluding remarks}
\label{sec:concluding}
\setnewcounter

We discuss two topics here. 

\subsection{Process limit}
\label{sec:process limit}

It is conjectured in \cite{Miya2024} that a process limit of the diffusion scaled queue length in the 2-level $GI/G/1$ queue in heavy traffic is the solution $Z(\cdot)$ of stochastic integral equation \eq{SIE-Z} for the 2-level SRBM. As we discussed in \rem{compatibility}, the stationary distribution of $Z(\cdot)$ is identical with the limit of the stationary distribution of the scaled queue length in the 2-level $GI/G/1$ queue in heavy traffic, obtained in \cite{Miya2024}. This strongly supports this conjecture on the process limit.

We believe that the conjecture is true. However, the standard proof for diffusion approximation based on functional central limit theorem may not work because of the state dependent arrivals and service speed in the 2-level $GI/G/1$ queue. We are now working on this problem by formulating the queue length process for the 2-level $GI/G/1$ queue as a semi-martingale. However, we have not yet completed its proof, so this is an open problem.

\subsection{Stationary distribution under weaker conditions}
\label{sec:weaker conditions}

In this paper, we derived the stationary distribution for a one-dimensional multi-level SRBM under the stability condition. In the view of \cor{general k}, it is naturally questioned whether a similar stationary distribution is obtained under more general conditions than \cond{multi-level}.

To consider this problem, we first need the existence of the solution $Z(\cdot)$ of \eq{SIE-Z}, for which the existence of the solution $X(\cdot)$ of \eq{SIE-X} is sufficient as discussed in \rem{existence}. For the latter existence, \cond{VD2} is weaker than \cond{multi-level}, but Theorem 5.15 of \cite{KaraShre1998} and Theorem 23.1 of \cite{Kall2001} show that it can be further weakened to
\begin{condition}
\label{cond:minimal VD}
\begin{align}
\label{eq:sigma positive}
 & \sigma^{2}(x) > 0, \qquad \forall x \in \dd{R},\\
\label{eq:local integrability 1}
 & \int_{x_{1}}^{x_{2}} \frac 1{\sigma^{2}(y)} dy < \infty, \qquad \forall (x_{1},x_{2}) \in \dd{R}^{2} \mbox{ satisfying } x_{1} < x_{2},\\
\label{eq:local integrability 2}
 & \int_{x-\varepsilon}^{x+\varepsilon} \frac {|b(y)|} {\sigma^{2}(y)} dy < \infty, \qquad \forall x \in \dd{R}, \exists \varepsilon > 0.
\end{align}
\end{condition}
It is easy to see that \cond{minimal VD} is indeed implied by \cond{VD2}. Note that the local integrability condition \eq{local integrability 1} implies that
\begin{align*}
  \lim_{\varepsilon \downarrow 0} \int_{x-\varepsilon}^{x+\varepsilon} \frac 1{\sigma^{2}(y)} dy < \infty, \qquad \forall x \in \dd{R},
\end{align*}
which is equivalent to $S_{\sigma} = \emptyset$, where
\begin{align*}
  S_{\sigma} = \left\{x \in \dd{R}; \lim_{\varepsilon \downarrow 0} \int_{x-\varepsilon}^{x+\varepsilon} \frac 1{\sigma^{2}(y)} dy = \infty\right\}.
\end{align*}
This condition $S_{\sigma} = \emptyset$ is needed for $X(t)$ to exist for all $t \ge 0$ in the weak sense as shown by Theorem 23.1 of \cite{Kall2001} and its subsequent discussions.

Assume \cond{minimal VD} for general $\sigma(x)$ and $b(x)$. If these functions are well approximated by simple functions (e.g., the discontinuity points of $\sigma(x)$ and $b(x)$ is finite for $x$ in each finite interval) and if $b(x) \not= 0$ for all $x \ge 0$, then \cor{general k} suggests that the stationary density is given by \eq{hx} under the condition that
\begin{align}
\label{eq:stability g}
  \int_{0}^{\infty} \frac {1}{\sigma^{2}(x)} \exp\left(\int_{0}^{x} \frac {2b(y)}{\sigma^{2}(y)} dy\right) dx < \infty.
\end{align}

To legitimize this suggestion, we need to carefully consider the approximation, but we have not yet done it. So, we leave it as a conjecture. 

\appendix

\section*{Appendix}
\setnewcounter
\setcounter{section}{1}

\subsection{Weak solution of a stochastic integral equation}
\label{app:weak-solution}

There are two kinds of solutions for a stochastic integral equation such as \eq{SIE-Z}. We here only consider them for the SIE \eq{SIE-Z}. Recall that this equation is defined on stochastic basis $(\Omega,\sr{F},\dd{F},\dd{P})$. On this stochastic basis, if \eq{SIE-Z} holds almost surely on this stochastic basis, then the SIE \eq{SIE-Z} is said to have a strong solution. In this case, the standard Brownian motion $W(\cdot)$ is defined on $(\Omega,\sr{F},\dd{F},\dd{P})$. On the other hand, the SIE \eq{SIE-Z} is said to have a weak solution if there are some stochastic basis $(\widetilde{\Omega},\widetilde{\sr{F}},\widetilde{\dd{F}},\widetilde{\dd{P}})$ and some $\widetilde{\dd{F}}$-adapted process $(Z(\cdot),W(\cdot),Y(\cdot))$ on it such that \eq{SIE-Z} holds almost surely and $W(\cdot)$ is the standard Brownian motion under $\widetilde{\dd{P}}_{x}$ for each $x \ge 0$, where $\widetilde{\dd{P}}_{x}$ is the conditional distribution of $\widetilde{\dd{P}}$ given $Z(0) = x$ (e.g., see \cite[Section 5.3]{KaraShre1998}).

It may be better to use a different notation for the weak solution, e.g., $(\widetilde{Z}(\cdot),\widetilde{W}(\cdot),\widetilde{Y}(\cdot))$. However, we have used the same notation not only for this process but also stochastic basis for notational convenience. Thus, when we discuss about the weak solution, the stochastic basis $(\Omega,\sr{F},\dd{F},\dd{P})$ is considered to be appropriately replaced.

\subsection{Local time and generalized Ito formula}
\label{app:local-time}

We briefly discuss about local time for a generalized Ito formula \eq{G-Ito}. This Ito formula is also called Ito-Meyer-Tanaka formula (e.g., see Theorem 6.22 of \cite{KaraShre1998} and Theorem 22.5 of \cite{Kall2001}). Let $X(\cdot)$ be a continuous semi-martingale with finite quadratic variations $[X]_{t}$ for all $t \ge 0$. For this $X(\cdot)$, local time $L_{x}(t)$ for $x \in \dd{R}$ and $t \ge 0$ is defined through
\begin{align}
\label{eq:L-time1}
  \int_{-\infty}^{\infty} L_{x}(t) g(x) dx = \int_{0}^{t} g(X(u)) d[X]_{u} \quad \mbox{ for any measurable function $g$}.
\end{align}
See Theorem 7.1 of \cite{KaraShre1998} for details about the definition of local time. Note that the local time of \cite{KaraShre1998} is half of the local time in this paper. Applying $g(y) = 1_{(x-\varepsilon,x+\varepsilon)}(y)$ for $\varepsilon > 0$ to \eq{L-time1}, we can see that
\begin{align}
\label{eq:L-time2}
  L_{x}(t) = \lim_{\varepsilon \downarrow 0} \frac 1{2\varepsilon} \int_{0}^{t} 1_{(x-\varepsilon,x+\varepsilon)}(X(u)) d[X]_{u}, \quad a.s. \quad x \in \dd{R}, t \ge 0
\end{align}
This can be used as the definition of the local time.

There are two versions of the local time since $L_{x}(t)$ is continuous in $t$, but may not be continuous in $x$. So, usually, the local time $L_{x}(t)$ is assumed to be right-continuous for the generalized Ito formula \eq{G-Ito}. However, if the finite variation component of $X(\cdot)$ is not atomic, then $L_{x}(t)$ is continuous in $x$ (see Theorem 22.4 of \cite{Kall2001}). In particular, the finite variation component of $Z(\cdot)$ is continuous by \lem{SIE-Z}, so we have the following lemma.

\begin{lemma}
\label{lem:LT-time}
For the $Z(\cdot)$ of an 1-dimensional state-dependent SRBM with bounded drifts, its local time $L_{x}(t)$ is continuous in $x$ for each $t \ge 0$. Furthermore, $\dd{E}[L_{x}(t)]$ is finite by \eq{L-time1} for $X(\cdot) = Z(\cdot)$.
\end{lemma}

Let $f$ be a concave test function from $\dd{R}$ to $\dd{R}$, then $-f$ is a convex function, where $(-f)(x) = - f(x)$, so the generalized Ito formula \eq{G-Ito} becomes
\begin{align}
\label{eq:G-Ito+concave}
  f(X(t)) = f(X(0)) + \int_{0}^{t} f'(X(u)-) dX(u) - \frac 12 \int_{0}^{\infty} L_{x}(t) \mu_{-f}(dx), \qquad t \ge 0,
\end{align}

For constant $a \in \dd{R}$, let $f(x) = (x-a)^{+} \equiv \max(0,x-a)$ for \eq{G-Ito}, then $f'(x-) = 1(x > a)$ and $\mu_{f}(B) = 1(a \in B)$. Hence, it follows from \eq{G-Ito} that
\begin{align}
\label{eq:Tanaka+}
 & (X(t) - a)^{+} = (X(0) - a)^{+} + \int_{0}^{t} 1(X(u)>a) dX(u) + \frac 12 L_{a}(t).
\end{align}
Similarly, applying $f(x) = (x-a)^{-} \equiv \max(0, -(x-a))$ and $f(x) = |x-a|$, we have, by \eq{G-Ito+concave} and \eq{G-Ito},
\begin{align}
\label{eq:Tanaka-}
 & (X(t)-a)^{-} = (X(0)-a)^{-} - \int_{0}^{t} 1(X(u) \le a) dX(u) + \frac 12 L_{a}(t), \\
\label{eq:Tanaka}
 & |X(t)-a| = |X(0)-a| + \int_{0}^{t} \sgn(X(u)- a) dX(u) + L_{a}(t), \qquad t \ge 0,
\end{align}
where $\sgn(x) = 1(x>0) - 1(x \le 0)$. Note that either one of these three formulas can be used to define local time $L_{a}(t)$. In particular, \eq{Tanaka} is called a Tanaka formula because it is originally studied for Brownian motion by \citet{Tana1963}. 

\section*{Acknowledgement} 

This study is originally motivated by the BAR approach, coined by Jim Dai (e.g., see \cite{BravDaiMiya2024}). I am grateful to him for his continuous support for my work. I am also grateful to Rami Atar for his helpful comments on the solution of stochastic integral equation \eq{SIE-Z}. I also benefit Rajeev Bhaskaran through personal communications. At last but not least, I sincerely thank Krishnamoorthy for encouraging me to present a talk at International Conference on Advances in Applied Probability and Stochastic Processes, held in Thrissur, India in January, 2024. This paper is written to follow up this talk and my paper \cite{Miya2024}.

%\bibliography{../../../texmf/bib/dai20230708M2}
%\end{document}

\def\cprime{$'$} \def\cprime{$'$} \def\cprime{$'$} \def\cprime{$'$}
  \def\cprime{$'$} \def\cprime{$'$} \def\cprime{$'$}

\end{document}